\documentclass[11pt, leqno]{amsart} 
\usepackage{amssymb,amscd,amsfonts,amsbsy}
\usepackage{latexsym}
\usepackage{exscale}
\usepackage{amsmath,amsthm,amsfonts}
\usepackage{mathrsfs}
\usepackage{xcolor} 
\usepackage[colorlinks=true,linkcolor=blue,citecolor=red,urlcolor=red, 
]{hyperref} 
\usepackage{esint} 
\usepackage{multienum} 
\usepackage{tasks} 
\usepackage{bm} 
\usepackage[bb=pazo]{mathalpha} 
\usepackage[utf8]{inputenc}

\usepackage{color}
\usepackage[active]{srcltx}

\calclayout

\allowdisplaybreaks

\theoremstyle{plain}
\newtheorem{theorem}[equation]{Theorem}
\newtheorem{lemma}[equation]{Lemma}
\newtheorem{corollary}[equation]{Corollary}

\newtheorem{claim}[equation]{Claim}

\theoremstyle{definition}
\newtheorem{definition}[equation]{Definition}

\theoremstyle{remark}
\newtheorem{remark}[equation]{Remark}

\numberwithin{equation}{section} 

\newcommand{\mipie}[1]{{\let\thefootnote\relax\footnotetext{\hspace{-.55cm}#1}}}

\newcommand{\quotes}[1]{``#1''}

\newcommand{\abs}[1]{\left\lvert #1 \right\rvert}
\newcommand{\pair}[2]{\langle #1, #2 \rangle}
\newcommand{\norm}[1]{\left\lVert #1 \right\rVert}

\newcommand{\re}{\mathbb{R}}
\newcommand{\rn}{\mathbb{R}^n}

\def\XXint#1#2#3{{\setbox0=\hbox{$#1{#2#3}{\int}$ }
	\vcenter{\hbox{$#2#3$ }}\kern-.585\wd0}}

\def\div{\operatorname{div}}
\def\tdiv{\widetilde{\div}}
\def\dist{\operatorname{dist}}

\def\Lip{\operatorname{Lip}}
\def\loc{{\operatorname{loc}}}

\def\supp{\operatorname{supp}}

\def\DD{\mathcal{D}}

\def\L{\widetilde{L}}

\def\NN{\mathcal{N}}
\def\RR{\mathcal{R}}

\def\R{\mathbb{R}}
\def\Rn{{\mathbb{R}^n}}
\def\re{\mathbb{R}}

\def\V{{\widetilde{V}}}

\def\Z{\mathbb{Z}}
\def\ttheta{\widetilde{\theta}}

\begin{document}

\title[The Kato problem for operators in non-divergence form]{On the Kato problem for elliptic operators in non-divergence form} 

\author{Luis Escauriaza}
\address{Luis Escauriaza\\ 
	Universidad del Pa\'is Vasco\\ Apartado 644, 48080 Bilbao, Spain.
}

\author{Pablo Hidalgo-Palencia}
\address{Pablo Hidalgo-Palencia\\
	Instituto de Ciencias Matemáticas CSIC-UAM-UC3M-UCM\\
	Con\-se\-jo Superior de Investigaciones Científicas\\
	E-28049 Ma\-drid, Spain
}
\address{and}
\address{Departamento de Análisis Matemático y Matemática Aplicada
	\\
	Facultad de Matemáticas
	\\
	Universidad Complutense de Madrid
	\\
	E-28040 Madrid, Spain
}
\email{pablo.hidalgo@icmat.es}

\author{Steve Hofmann}
\address{Steve Hofmann\\ 
	Department of Mathematics, University of Missouri, Columbia, MO 65211, USA
}
\email{hofmanns@missouri.edu}

\thanks{The second author is supported by the grant CEX2019-000904-S-20-3, funded by MCIN/AEI/ 10.13039/501100011033, and acknowledges financial support from MCIN/AEI/ 10.13039/501100011033 grants CEX2019-000904-S and PID2019-107914GB-I00. The third
author was supported by NSF grant DMS-2000048. Part of this work was carried out while the first and third authors were visiting ICMAT in Madrid, and part of this work was carried out while
the second author was visiting the third author at the University of Missouri - Columbia. The authors express their gratitude to these institutions.}

\date{\today}

\dedicatory{Dedicated to Prof. Carlos Kenig on the occasion of his 70th birthday, and to the memory of Luis Escauriaza.}

\makeatletter
\@namedef{subjclassname@2020}{\textup{2020} Mathematics Subject Classification}
\makeatother




\begin{abstract}
We consider the Kato square root problem for non-divergence second order elliptic operators $L =- \sum_{i,j=1}^{n}a_{ij} D_iD_j$, and, especially, the normalized adjoints of such operators.  In particular, our results are applicable to the case of real coefficients having sufficiently small BMO norm.  We assume that the coefficients of the operator are smooth, but our quantitative estimates do not depend on the assumption of smoothness. 
\end{abstract}

\maketitle

\tableofcontents


\section{Introduction and main result} \label{sec:introduction}

In \cite{AHLMcT}, the authors resolved in the affirmative the long-standing square root problem of Kato for divergence form complex-elliptic operators in $\rn$.  This was the culmination of a series of previous results: \cite{CMcM} (treating the 1-d case); \cite{CDM} and \cite{FJK1,FJK2} (treating small perturbations of the constant coefficient case); \cite{HMc} (the 2-d case); \cite{AHLT} (small perturbations of the real-symmetric case, sometimes referred to as the restricted Kato problem);
and \cite{HLMc} (the case that the heat kernel satisfies a Gaussian upper bound and Nash-type H\"older continuity estimates).   The solution of the Kato problem, and the circle of ideas involved in its proof, led to subsequent breakthroughs in the theory of elliptic boundary value problems, see, e.g., \cite{AAAHK}, \cite{Hof} \cite{HKMP1,HKMP2}, \cite{AKMc}, \cite{AA}, \cite{AAH}, \cite{AAMc,AAMc2}, \cite{CUR}, \cite{HLeM}, \cite{EH}, \cite{HLMP1,HLMP2}, \cite{BHLMP}. 
See also the significant related ground-breaking work in the parabolic setting: \cite{Ny}, \cite{AEN1,AEN2}.

In this note, we initiate the study of the square root problem 
in the non-divergence setting in dimensions greater than 1, with the eventual goal of
developing applications to the theory of boundary value problems, as has been done in the aforementioned divergence form case.  Previously, the non-divergence problem had been treated only in the 1-dimensional setting,
in \cite{KM}; in fact, in that paper the authors treat the more general class of operators of the form $L=-aDbD$, where $D$ denotes the ordinary differentiation operator on the line, and $a,b$ are arbitrary bounded accretive complex-valued functions on the line.  

At present, we are able to treat only the case of real coefficients.  On the other hand, we point out that in the divergence form case, there is no known proof for real, non-symmetric coefficients
that is fundamentally easier than the proof in the general 
case, owing to the non-selfadjointness of non-symmetric 
divergence form operators.  
Moreover, it is the real, non-symmetric case that underlies the breakthrough in the study of the Dirichlet  problem obtained in
\cite{HKMP1}.  We observe that in the non-divergence setting, we may assume without loss of generality that the coefficient matrix is symmetric, but in contrast to the divergence form case, operators of non-divergence type are inherently non-selfadjoint, even with symmetric coefficients.

A fundamental difficulty that one encounters in the non-divergence setting, is that the Kato problem
for non-divergence elliptic operators seems to be most naturally formulated in a weighted $L^2$ space, and in general, the weight need not belong to the Muckenhoupt $A_2$ class.  
Another difficulty inherent to the non-divergence setting is the lack of uniqueness (see the work of Nadirashvili \cite{N}).
 On the other hand, working with real coefficients allows us to make use of the pioneering work of
 Krylov and Safonov \cite{KS}, as well as the important ideas of Baumann \cite{B} and 
  Escauriaza \cite{E}.  We shall return to these matters in the sequel.  First, let us set notation and definitions.

We will say that the operator $L$ is a \textit{second-order elliptic operator in non-divergence form} on $\R^n$ if
\begin{equation} \label{Ldef}
Lu = - \sum_{i, j = 1}^n a_{ij} D_i D_j u,
\end{equation}
where $A = (a_{ij}(\cdot))$ is a 
real and measurable coefficient matrix which (without loss of generality) we can take to be symmetric, and for which we also assume, for some $\lambda > 0$,
\begin{equation} \label{eq:ellipticity}
A(x) \xi \cdot \xi \geq \lambda \abs{\xi}^2
\quad \text{and} \quad 
\abs{A(x) \xi \cdot \zeta} \leq \lambda^{-1} \abs{\xi} \abs{\zeta} 
\qquad \text{for all } \xi, \zeta \in \Rn \text{  and a.e. } x \in \Rn ,
\end{equation}


For such $L$, we have also its adjoint operator 
\begin{equation}\label{L*def}
L^* u = - \sum_{i, j = 1}^n D_i D_j (a_{ij} u).
\end{equation}
Following \cite{E}, we say that the function $u \in L^1_\loc(\Rn)$ is a solution of the adjoint equation $L^*u = 0$ if for every $\varphi \in \mathscr{C}_c^\infty(\Rn)$ we have 
\begin{equation*}
\int_\Rn u(x) L\varphi(x) dx = 0.
\end{equation*}


Let us recall also the definition of Muckenhoupt weights, that will be used throughout the text because of the properties of some particularly relevant adjoint solution, as we shall see shortly.

\begin{definition}[Muckenhoupt weights] \label{def:A_p}
	We say that the function $w$ belongs to the Muckenhoupt class of weights $A_p$ for some $1 < p < \infty$ if $w(x) > 0$ a.e. $x \in \Rn$ and 
	\begin{equation*}
	[w]_{A_p} := \sup_B \left(\frac{1}{\abs{B}} \int_B w(x) dx \right) \left(\frac{1}{\abs{B}} \int_B w(x)^{1-p'} dx\right)^{p-1} < \infty,
	\end{equation*} 
	where the supremum is taken over all the balls $B \subset \Rn$, and also denote $A_\infty := \bigcup_{1 < p < \infty} A_p$. 
\end{definition}
We recall that it is well known that the $A_\infty$ property is equivalent to the Reverse H\"older property, i.e., that there is an exponent $q>1$, and a uniform constant $C$ such that for every ball $B$,
\[
\left(\frac{1}{\abs{B}} \int_B w^q(x) dx \right)^{1/q} \leq C \frac{1}{\abs{B}} \int_B w(x) dx  
\leqno(RH_q)
\]
For details of the theory of Muckenhoupt weights, the reader may consult, e.g. \cite[Chapter 7]{Duo}.

With this definition in mind, we recall a fundamentally important property of equations in non-divergence form.

\begin{lemma}[{\cite[Theorem 1.1]{E}}] \label{lem:existence_W}
	Let $L$ be a second-order elliptic operator in non-divergence form, and $L^*$ its adjoint. Then there exists a non-negative solution $W$ of the adjoint equation $L^*W = 0$ in $\R^n$, satisfying $W(B_1(0)) = \abs{B_1(0)}$, which we call the {\tt global non-negative adjoint solution}. Furthermore, $W$ satisfies a Reverse Hölder property with exponent $\frac{n}{n-1}$, so $W \in A_\infty$ (c.f. Definition~\ref{def:A_p}).  Moreover, the $RH_{n/(n-1)}$ constants depend only on dimension and ellipticity.
\end{lemma} 
If the coefficients of $L$ are smooth, or even belong to VMO, then $W$ (with the stated normalization) is unique.  In general, it need not be unique.  
On the other hand, for any given $L$, any choice of such a $W$ will enjoy the 
same quantitative estimates, with 
uniform control of all relevant constants.  In the case that $W$ is not 
unique, we may therefore simply fix an arbitrary choice of $W$.

It is a well known fact that $A_\infty$ weights are doubling. In our case, this means that there exists a constant $C_D = C_D([W]_{A_\infty}) \geq 1$ such that $W(2B) \leq C_D W(B)$ for every ball $B$.

From now on we will work most of the time in the weighted Hilbert space 
$$L^2_W := L^2(\R^n, W(x)dx).$$ 
In the non-divergence setting, this space is more natural in many ways than  unweighted $L^2$. In particular, the following identity holds, as may be seen formally by using $L^*W = 0$ and integrating by parts:
\begin{equation} \label{eq:Luu}
\int_\Rn u Lu W dx = \int_\Rn A \nabla u \cdot \nabla u W dx.
\end{equation}
In fact, one may readily deduce that \eqref{eq:Luu} holds when the coefficients are smooth, and more generally, it also holds at least 
when the coefficients have sufficiently small BMO norm (depending only on dimension and ellipticity), for all $u\in \mathcal{D}(L)$ (the domain of $L$), defined by 
\[\mathcal{D}(L):= \left\{u\in L^2_W: Lu \in L^2_W\right\}.
\]
Indeed,  if the coefficients 
have sufficiently small BMO norm, then 
 $W\in A_2$ (see \cite[Theorem 1.2]{E1}, and its proof\footnote{In fact, \cite{E1} precedes \cite{E} chronologically; the result in \cite{E1} treats explicitly local versions of $W$, but the same arguments may be applied to the global $W$ constructed in \cite{E}.}).  In turn, 
 using this fact, 
one may prove\footnote{We caution the reader that the proof of \eqref{22reg} is a somewhat non-trivial matter, and is based on the boundedness of the commutator $[T,b]$, where $T$ is a singular integral and $b\in$ BMO \cite{CRW}, along with a suitable expansion in terms of spherical harmonics; see \cite{CFL1,CFL2} for related results in the unweighted case.} the regularity estimate
\begin{equation}\label{22reg}
\|\nabla^2 u\|_{L^2_W} \lesssim \|f\|_{L^2_W}
\end{equation}
 for solutions of the Poisson problem $Lu=f\in L^2_W$, 
and hence, that 
\begin{equation}\label{domain22}
\mathcal{D}(L) = H_W^{2}(\rn) := \left\{u\in L^2_W: \nabla u, \nabla ^2 u \in L^2_W\right\}.
\end{equation}
The identity  \eqref{eq:Luu} then follows readily for all $u\in \mathcal{D}(L)$.

\begin{remark}\label{rdensity}
We observe that $H_W^{2}(\rn)$ is dense in $L^2_W$ 
when $W\in A_2$ (indeed, even $\mathscr{C}_c^\infty$ is dense in that case), and therefore $L$ is densely defined when the 
coefficients have sufficiently small BMO norm.
\end{remark}

 We shall also consider the {\em normalized adjoint} of $L$,
  which we denote by $\L$, and which we define
  to be the adjoint of $L$ with respect 
  to the space $L^2_W$.  Thus, $\L$
is given, at least for smooth coefficients, by the formula
\begin{equation} \label{Ltildedef}
\L u := - \frac1W \sum_{i, j = 1}^n D_i D_j (a_{ij} u W) = \frac1W\, L^*(uW).
\end{equation}
If the coefficients are merely measurable, 
then we interpret $\L$ in the weak sense: 
we say that  $u\in L^2_W$ belongs to 
$\mathcal{D}(\L)$, the domain of $\L$, provided that there is a function
$g\in L^2_W$ such that for 
every $v\in \mathcal{D}(L)$,
\[
\int_{\rn} Lv \, u\, W dx = \int_{\rn} v g W dx\,,
\]
and in this case we set $\L u = g$.
Just as for \eqref{eq:Luu}, integrating by parts and using $L^*W = 0$, 
we obtain (at least in the case of smooth coefficients, and for $u\in H^2_W$)
\begin{equation} \label{eq:adjLuu}
\int_\Rn u \L u W dx = \int_\Rn A \nabla u \cdot \nabla u W dx.
\end{equation}


We shall henceforth make the {\em qualitative} assumption
that the coefficients $a_{ij}$ are 
smooth, with qualitative $L^\infty$ bounds on $\nabla a_{ij}$ and 
$\nabla^2 a_{ij}$.  Thus, \eqref{eq:adjLuu} will be valid in the sequel,
for $u\in H^2_W$.
On the other hand, we emphasize that our {\em quantitative} 
bounds will never
depend on smoothness, nor on estimates for the derivatives of $a_{ij}$.

\begin{remark}\label{functional}
Using \eqref{eq:Luu}, \eqref{eq:adjLuu}, \eqref{22reg}, and \eqref{domain22}\footnote{\eqref{domain22} 
is used to show that
\eqref{eq:Luu} holds for all $u \in \mathcal{D}(L)$.}, one may then 
show that $L$, and $\L$, viewed as unbounded operators on $L^2_W$, 
are each closed, sectorial and m-accretive, and hence each has an 
m-accretive square root
(see \cite[Theorem 3.35, p. 281]{Ka}, or \cite[Sections 3 and 7]{H}). 
Moreover, $L$
generates a heat semigroup
$z\mapsto e^{-zL}$, which is well-defined and analytic in a sector containing 
the positive real axis (hence, the analogous statement is also true for $\L$).
\end{remark}
Let us now sketch the proofs of the functional analytic facts listed in
Remark \ref{functional}.

\noindent {$L$ and $\L$ \it are closed operators}.  For $\L$, this 
follows immediately from the fact 
that $L$ is densely defined (see \cite[p. 168]{Ka}).
Thus we consider $L$.
Suppose that $\{u_n\}_n \subset \DD(L)$, 
that $u_n \to u$ in $L^2_W$, and that 
\[
f_n:= L u_n \to f \text{ in } L^2_W\,.
\]
We need to verify that $u\in \DD(L)$, and that $L u=f$, i.e., that the graph
$\{(u,Lu):u \in \DD(L)\}$ is a closed set in $L^2_W\times L^2_W$.
Applying \eqref{eq:Luu} and \eqref{22reg} to $u_n-u_m$, we see that 
$\{\nabla u_n\}_n$ and $\{\nabla^2 u_n\}_n$ are each convergent in $L^2_W$, 
thus, $\{u_n\}_n$ is convergent in $H^2_W$, and since $u_n\to u$ in $L^2_W$, we see that $u\in H^2_W =\DD(L)$, and that $u_n \to u$ in $H^2_W$.  In particular,
$D_iD_j u_n \to D_iD_j u$ in $L^2_W$ for each $i,j = 1,2,...,n$, hence
$Lu_n \to Lu$, so that $Lu=f$, as desired.

\noindent {$L$ and $\L$ \it are sectorial}.  It follows readily from
\eqref{eq:Luu} (respectively, \eqref{eq:adjLuu}) that the numerical ranges
\[
\Theta:=\left\{\langle Lu,u\rangle \in \mathbb{C}: \|u\|_{L^2_W} =1\right\}\,,
\quad
\widetilde{\Theta}:=
\left\{\langle \L u,u\rangle \in \mathbb{C}: \|u\|_{L^2_W} =1\right\}
\]
are each contained in a sector 
$S_\omega:= \{z\in\mathbb{C}: |\arg z|\leq \omega\} \cup \{0\}$,
with $0<\omega<\pi/2$, depending only on ellipticity.
We omit the standard argument.

\noindent {$L$ and $\L$ \it are m-accretive}.
By \cite[Problem 3.31, p 279]{Ka}, it suffices to show that $L$ is 
m-accretive.  To this end, set 
\[
\Delta:= \mathbb{C}\setminus S_\omega\,,
\]
and for $\zeta \in \Delta$, set $\delta=\delta(\zeta):=
\dist(\zeta,S_\omega)$.  By symmetry, we also have
$\delta = \dist(\overline{\zeta},S_\omega)$. 
We now claim that 
\begin{equation}\label{resolventbound}
\|(L-\zeta)u\|_{L^2_W} \geq \delta \|u\|_{L^2_W}\,,\quad 
u \in \DD(L)\,,\,\, \zeta \in \Delta\,.
\end{equation}
Indeed, to verify the claim, we may assume without loss of generality 
that $\|u\|_{L^2_W} = 1$,
in which case
\[
\|(L-\zeta)u\|_{L^2_W} \geq \left|\langle (L-\zeta)u,u\rangle  \right| 
= \left|\langle Lu,u\rangle - \zeta \right| \geq \delta\,,
\]
since $\langle Lu,u\rangle \in \Theta \subset S_\omega$.

Fix $\zeta \in \Delta$.  Then
 $L-\zeta$ is 1-1 on $\DD(L)$, and has closed range 
(since $L$ is a closed operator).
  Similarly, $\L-\overline{\zeta}$ is 1-1 on $\DD(\L)$. 
  Since $L$ is densely defined,
\[
\NN(\L-\overline{\zeta}) = \RR(L-\zeta)^\perp\,,
\]  
i.e., the null space of $\L-\overline{\zeta}$ 
is the orthogonal complement of the range of
$L-\zeta$. Thus, $L-\zeta$ has dense range, since 
$\L-\overline{\zeta}$ is 1-1. Hence,
$L-\zeta$ is invertible as a mapping from $\DD(L)$ onto $L^2_W$.  
Combined with the estimate \eqref{resolventbound}, this shows that 
$L$ is m-accretive 
(see \cite[p. 279]{Ka}).

\noindent {\it The heat semi-group}.
Given the preceeding properties of $L$, we have 
existence, uniqueness, and analyticity of a contraction semigroup
$e^{-zL}$, for $z$ in the open sector $S^0_\alpha
:= \{z\in\mathbb{C}: |\arg z| < \alpha\}$, 
provided $0<\alpha < \pi/2-\omega$.  See, e.g., \cite[pp. 480-493, especially
Theorem 1.24, p. 492]{Ka}.


Our main result is the following:
\begin{theorem} \label{th:kato}
	Let $L$ be a second-order elliptic operator in non-divergence form with smooth real coefficients 
satisfying \eqref{eq:ellipticity}, and let $W$ be the 
associated global non-negative adjoint solution provided by Lemma~\ref{lem:existence_W}. If $W \in A_2$ (see
Definition~\ref{def:A_p}), then we have
	\begin{equation*}
	\norm{\sqrt{L}f}_{L^2_W} \approx \norm{\nabla f}_{L^2_W}
	\approx \norm{\sqrt{\L}f}_{L^2_W}, 
	\end{equation*}
	where the implicit constants depend only on $n, \lambda$ and 
	$[W]_{A_2}$. 
	\end{theorem}

The main goal of this paper is to prove this Theorem. 
Hence, from here on we will always impose the 
extra assumption that $W \in A_2$, along with the qualitative assumption
that the coefficients are smooth. The result will follow at once from Theorems~\ref{th:kato_L} and \ref{th:kato_adj}. 

Some additional remarks are in order.

\begin{remark}
As mentioned above, in general $W$ belongs to the class 
$A_\infty$, and thus $W\in A_p$ for some $p$ depending on dimension
and ellipticity, but $p$ may be strictly greater than 2, and in 
fact in the general case we have no precise upper bound on $p$.  
Thus, our
result is only a partial one, and does not address the fundamental challenge of treating the non-$A_2$ case.  
On the other hand, as noted above, if the coefficients 
have sufficiently small BMO norm, then 
 $W\in A_2$, 
 and thus our result does apply in that setting.
 \end{remark}
 
 \begin{remark}\label{remark:L-easy} In the case that the coefficient matrix has sufficiently small BMO norm, then as also noted above, we may identify the domain 
 $\mathcal{D}(L)$ as the weighted Sobolev space $H^2_W(\rn)$ (see \eqref{domain22}). Hence, by combining several known (or at least implicit) results, we may identify the domain of $\sqrt{L}$ as the Sobolev space $H_W^1(\rn):=\{u\in L^2_W(\rn): \nabla u \in L^2_W(\rn)\}$; this corresponds to the estimate $\norm{\sqrt{L}f}_{L^2_W} \lesssim \norm{\nabla f}_{L^2_W}$.
 Indeed, in \cite{DY} it is shown that the operator $L$ has a bounded holomorphic functional calculus in (unweighted) $L^2$ (even in $L^p$), provided that the BMO norm of the coefficients is sufficiently small.  Under the same smallness assumption, the arguments of \cite{DY} may be extended to the weighted case considered here, to deduce that $L$ has a bounded holomorphic functional calculus in $L_W^2$.  Combining the results of \cite{Y} and \cite{Mc}, we find that $\mathcal{D}(\sqrt{L})$ is the complex interpolation space mid-way between $L^2_W$ and 
 $\mathcal{D}(L) = H^2_W$, i.e.,  $\mathcal{D}(\sqrt{L}) = H^1_W$.
 
 The analogous strategy fails for $\L$, as we have no idea 
 how to identify 
 $\mathcal{D}(\L)$ (similarly, the square root problem in the divergence form case entailed the same difficulty).
\end{remark}

\begin{remark} The assumption of smoothness of the coefficients is purely qualitative, and our quantitative estimates will not depend on smoothness, but only on the stated parameters $n, \lambda$ and 
$[W]_{A_2}$.   However, it is not clear at 
present how to make sense of the identity
\eqref{eq:adjLuu} for non-smooth coefficients, and as a consequence,
 in the absence of smoothness, 
 we do not know how to prove certain estimates 
which rely on \eqref{eq:adjLuu}, such as 
Lemma \ref{lem:bounded} \ref{bounded:gradient_adjoint_semigroup}
(in the case of measurable coefficients,
we know how to give only a formal proof of the latter, 
assuming a priori finiteness of $\|\nabla e^{-t^2\L} f\|_{L^2_W}$).

On the other hand, as noted above, identity \eqref{eq:Luu} holds without smoothness,
in the case that the coefficients have sufficiently small BMO norm.  
Under the latter scenario,
we require identity \eqref{eq:adjLuu} and its consequences
(and thus, the qualitative, a priori assumption of 
smoothness of the coefficients) in two places: 1)
in the proof of Theorem
\ref{th:kato_adj} (the square root problem for $\L$), where estimate 
\eqref{eq:adjLuu} is heavily used,
and 2) in the proof of the m-accretivity of $L$ given above, where we used 
\eqref{eq:adjLuu} to establish density of the range of $L-\zeta$.
Otherwise, \eqref{eq:adjLuu} is not used
in the proof of Theorem \ref{th:kato_L} (the 
square root problem of $L$). 
\end{remark}



Although our operators $L$ and $\L$ are not self-adjoint, 
there is a nice identity relating these two non-divergence operators with another one which is in divergence form, but degenerate
elliptic\footnote{Thus, our results here are somewhat related to those of Cruz-Uribe and Rios \cite{CUR}.}.
Indeed, if we let $\widetilde{\div}$ denote the normalized divergence, defined for an $\rn$-valued function ${\bf v}$
 by $\widetilde{\div}\, {\bf v}:= \frac{1}{W} \div (W{\bf v})$, then
$\widetilde{\div}$ is precisely the adjoint operator to $-\nabla$
inside $L^2_W$,  and we also have, using $L^*W=0$,
\begin{equation} \label{eq:Lu+adjLu}
Lu + \L u = -2 \widetilde{\div} (A \nabla u).
\end{equation}
In the case of non-smooth coefficients, we 
interpret the latter identity in the weak sense described above:  
for $u, \varphi \in H_W^2$,
\[
\int_{\rn} \big(\varphi Lu \,+\, u L\varphi\big) W dx =
2\int_{\rn} A\nabla u\cdot \nabla \varphi \,W dx\,.
\]
The identity \eqref{eq:Lu+adjLu}
will be of great use to us in the sequel.

The paper is organized as follows:
\begin{itemize} 
	\item In Section~\ref{sec:preliminaries} we give some definitions and estimates for some of the operators that which will appear repeatedly across the paper. 
	
	\item In Section~\ref{sec:kato_L} we prove $\norm{\sqrt{L}f}_{L^2_W} \lesssim \norm{\nabla f}_{L^2_W}$, which turns out 
to be a relatively easy consequence of Littlewood-Paley theory because of the form of $L$ (since $L$ annihilates not only constants but also first degree monomials).
	
	\item In Section~\ref{sec:kato_adj} we prove $\norm{\sqrt{\L}f}_{L^2_W} \lesssim \norm{\nabla f}_{L^2_W}$, which is in fact the more difficult result in the paper.  
To treat $\L$, we follow broadly the scheme provided by \cite{AHLMcT}, 
first reducing the problem to some square function estimates, which are handled using a $T1$-like argument and then a local $Tb$ argument. Of course, some significant modifications of the arguments in \cite{AHLMcT} are needed;  the identity \eqref{eq:Lu+adjLu} will be useful in this case.
\end{itemize}
We remark that the square root problem for $\L$ is significantly
more difficult 
than its analogue for $L$.


\subsection{Notation}

\begin{itemize}
	\item We use the notation $a \lesssim b$ to denote that there exists a positive harmless constant $C$ (which can vary from line to line) such that $a \leq Cb$. We will also denote $a \approx b$ whenever $a \lesssim b$ and $b \lesssim a$.
	
	\item In the proofs of the results from now on, we will omit dependencies of constants on $n$, $\lambda$, $[W]_{A_\infty}$ and $[W]_{A_2}$ -- treating them as harmless constants -- although we will make these dependencies explicit in the statements.
	
	\item Euclidean balls are denoted by $B_t(x) := \{ y \in \Rn : \abs{y-x} < t \}$.
	
	\item If $B = B_t(x) \subset \Rn$ is a ball and $\kappa > 0$, by $\kappa B$ we denote the ball with same radius and scaled by a factor of $\kappa$, i.e., $B_{\kappa t}(x)$. The same applies to cubes.
	
	\item For $E \subset \Rn$, $\abs{E}$ denotes the Lebesgue measure of $E$.
	
	\item If $E, F \subset \Rn$ are arbitrary subsets, we write $\dist(E, F) := \inf \{\abs{x-y} : x \in E, y \in F\}$.
	
	\item For any subset $E \subset \Rn$, we denote $\mathbf{1}_E$ the characteristic function of $E$ (i.e. $\mathbf{1}_E(x) = 1$ if $x \in E$ and 0 otherwise). Concretely, we write $\mathbf{1} := \mathbf{1}_\Rn$, the function constantly 1.
	
	\item We will denote vector-valued functions with boldface letters, e.g.,
	${\bf v} := (v_1,...,v_n).$	
	
	\item $D_j$ denotes the differentiation operator in the direction of $x_j$, i.e., $D_j = \frac{\partial}{\partial x_j}$.
	
	\item We denote averages respect to a measure $\nu$ by $\fint_E f d\nu := \nu(E)^{-1} \int_E f d\nu$. Often the measure with respect to which we take averages will be the weighted measure $W(x)dx$: it will be clear by the context.  For the latter measure, we write $W(E) := \int_E W(x) dx$.
	
	\item We will frequently use cubes in our proofs: every time we cover $\Rn$ (or some portion of it) by cubes, we mean we are using a covering by cubes of the dyadic grid $\{ 2^j \mathbf{k} + [0, 2^{j})^n : j \in \Z, \mathbf{k} \in \Z^n \}$. Anytime we use the letter $Q$, we will be referring to a dyadic cube. For such a cube $Q$, we let $\ell(Q)$ denote its sidelength.
	
	\item We let $\mathcal{M}$ and $\mathcal{M}_W$ denote, respectively,
the classical Hardy-Littlewood maximal operator, and the Hardy-Littlewood maximal function with respect to the measure $W(x)dx$, that is, 
\begin{equation*}
	\mathcal{M} f (x) 
	:= 
	\sup_{B \ni x} \fint_B \abs{f(y)} dy
	=
	\sup_{B \ni x} \frac{1}{|B|} \int_B \abs{f(y)} dy\,,
	\end{equation*}
and
	\begin{equation*}
	\mathcal{M}_W f (x) 
	:= 
	\sup_{B \ni x} \fint_B \abs{f(y)} W(y)dy
	=
	\sup_{B \ni x} \frac{1}{W(B)} \int_B \abs{f(y)} W(y)dy.
	\end{equation*}
	Since $W$ is doubling (because $W \in A_\infty$), $\mathcal{M}_W$ 
is bounded on $L^p_W$ for every $1 < p < \infty$. We will use this fact
in the sequel. 
	
	\item As explained before, we set
	$L^2_W := L^2(\Rn, W(x)dx)$, and we define the weighted Sobolev space $H_W^2:=\{u\in L^2_W: \nabla u\in L_W^2, \, \nabla^2 u \in L_W^2\}$. 
	We will also write
	$L_W^2(E) := L^2(E, W(x)dx)$ for any subset $E \subset \Rn$.	
	
	\item 
	We denote the composition of two operators $U$ and $V$ by 
	$UV(f) := U(V(f))$). 
	
	\item For a function $f\in L^2(\rn)$, we denote its Fourier transform by
	$\hat{f}$.
	\item $\mathcal{S}$ will denote the usual Schwarz class of smooth,
	rapidly decaying functions on $\rn$.
\end{itemize}

\noindent
{\bf Acknowledgements}.  The third author thanks Prof. X. T. Duong for an interesting conversation concerning the latter's joint work with L. Yan \cite{DY}, and in particular for pointing out to us the argument sketched in 
Remark \ref{remark:L-easy}.


\section{Preliminaries} \label{sec:preliminaries}

\subsection{Gaussian bounds for kernels of semigroups}

From now on, we will use many times the parabolic semigroup (with elliptic homogeneity) $e^{-t^2L}$, whose kernel is the fundamental solution $\Gamma_{t^2}(x, y)$;
i.e. we have $e^{-t^2L}f(x) = \int_\Rn \Gamma_{t^2}(x, y)f(y)dy$ for sufficiently regular $f$. 
The fundamental solution
satisfies the following Gaussian estimate:

\begin{lemma}[{\cite[Theorem 1.2]{E}}] \label{lem:fund_sol_bound}
	The kernel $\Gamma_{t^2}(\cdot, \cdot)$ of $e^{-t^2L}$ satisfies the Gaussian bounds
	\begin{equation} \label{eq:gaussian_bound}
	\Gamma_{t^2}(x, y) \lesssim \min \left\{ \frac{1}{W(B_t(x))}, \frac{1}{W(B_t(y))} \right\} e^{-c\frac{\abs{x-y}^2}{t^2}} W(y),
	\end{equation}
	and
	\begin{equation} \label{eq:gaussian_bound2}
	  \max \left\{ \frac{1}{W(B_t(x))}, \frac{1}{W(B_t(y))} \right\} e^{-\frac{\abs{x-y}^2}{ct^2}} W(y) \lesssim \Gamma_{t^2}(x, y),
	\end{equation}
	where the implicit constants and $c$ depend on $n$ and $\lambda$.
\end{lemma}

\begin{remark}\label{r2.4} The results stated above as
Lemma \ref{lem:existence_W} and Lemma \ref{lem:fund_sol_bound}, are stated in \cite{E} explicitly for smooth coefficients, but as the author points out, ``the usual compactness arguments",  
and the uniformity of the estimates depending only on $n$ and $\lambda$, allow one to deduce the existence of (non-unique)
$W$ and $\Gamma$ verifying the same bounds, in the general case of bounded measurable coefficients.
\end{remark}

\begin{remark}\label{r2.5} 
The doubling property of $W$, combined with the exponential decay factor, allow us to interchange ``min" and ``max" in \eqref{eq:gaussian_bound}
and \eqref{eq:gaussian_bound2}, modulo an adjustment of the constants.
\end{remark}


\begin{remark}\label{r2.6} The (absolute values of) the kernels of the operators $t^2Le^{-t^2L}$ and $t^4L^2e^{-t^2L}$ also satisfy the upper bound~\eqref{eq:gaussian_bound}, by analyticity of the semigroup
$z\mapsto e^{-zL}$ in a sector. 
\end{remark}




\subsection{Weighted Littlewood-Paley theory}

The following results are standard.  We recall them here for the sake of convenience of exposition.  


\begin{lemma} \label{lem:P_t_bounded} Let $W\in A_2$, and
	let $K_t f := k_t * f$, with $k$ defined on $\rn$ satisfying $|k(x)| \leq (1+|x|)^{-n-1}$, 
	where $k_t(x) := t^{-n}k(x/t)$. 
Then 
	\begin{equation*}
\|\sup_{t>0} |K_t f|\|_{L_W^2} \lesssim \|\mathcal{M} f\|_{L_W^2} \lesssim
 \|f\|_{L_W^2},	\end{equation*}
	where $\mathcal{M}$ is the classical Hardy-Littlewood maximal operator, 
and the implicit constant depends on $n$ and $[W]_{A_2}$.
\end{lemma}

\begin{lemma} \label{lem:Q_s_square_fun}
Let $W\in A_2$, and let
$Q_s f := \psi_s * f$, where $\psi \in \mathcal{S}$ and satisfies $\int_{\Rn} \psi 
= 0$. 
 Then 
	\begin{equation*}
	\int_0^\infty \norm{ Q_s f }_{L^2_W}^2 \frac{ds}{s} 
	\lesssim 
	\norm{f}_{L^2_W}^2,
	\end{equation*}
	where the implicit constant depends on $n, \psi$,
	and $[W]_{A_2}$.  
Moreover, if in addition $\psi$ is 
radial and non-trivial, then
using a slight abuse of notation and then normalizing, 
we may assume that  
$ \int_0^\infty \hat{\psi}(s)^2 \,\frac{ds}{s}  =1,$
in which case we have the Calder\'on reproducing formula
\[
\int_0^\infty Q_s^2 f\, \frac{ds}{s} = f \in L^2_W\,.
\]
\end{lemma}
\begin{remark} \label{rem:derivative_P_t}
	Regarding the last pair of lemmata:
	\begin{itemize}
		\item We will often denote by $Q_t$ the operators satisfying the hypotheses in Lemma~\ref{lem:Q_s_square_fun}, while we use $P_t$ for ``nice" approximate identities (i.e. $P_t f := \varphi_t * f$, with 
$\varphi \in \mathcal{S}$ radial,  and $\int \varphi = 1$). 
		\item It is easy to check if $P_t$ is a nice approximate identity, then $Q_t := t D_i P_t$ satisfies the hypotheses of the first part of Lemma~\ref{lem:Q_s_square_fun} (where $D_i$ denotes the partial
		derivative in any direction $x_i$).
		\item  We will frequently further assume 
that the kernel $k$ in Lemma \ref{lem:P_t_bounded} (in particular,  
$\varphi$ and $\psi$ as above) satisfies
$\supp k \subset B_1(0)$; in this case, we 
shall refer to $K_t$ (in particular, $P_t$ or $Q_t$) as having a ``compactly 
supported kernel".
		\item We will use repeatedly the fact that $P_t$ and $Q_s$ commute with derivatives, for they are convolution operators.
	\end{itemize}
\end{remark}

The following is an easy consequence of Lemma~\ref{lem:Q_s_square_fun}, by standard ``almost-orthogonality" arguments.  We omit the well-known proof.

\begin{lemma} \label{lem:orthogonality}
	Let $\{Q_s\}_{s>0}$ be a family operators satisfying the conditions in Lemma~\ref{lem:Q_s_square_fun}, and $R_t$ be a 
family of operators, bounded on $L^2_W(\rn)$ for each fixed $t>0$,
	and satisfying, for some $\alpha > 0$, the almost-orthogonality condition
	\begin{equation*}
	\norm{R_tQ_s}_{L^2_W \to L^2_W} \leq C_1 \min \left\{ \frac{t}{s}, \frac{s}{t} \right\} ^\alpha,
	\end{equation*}
	where $C_1$ is a uniform constant which does 
	not depend on $t, s$. Then $R_t$ satisfies the square function estimate
	\begin{equation*}
\int_0^\infty \norm{R_t f}_{L^2_W}^2 \frac{dt}{t} \lesssim \norm{f}_{L^2_W}^2,
	\end{equation*}
	where the implicit 
constant depends on $n, C_1, \alpha$ and $[W]_{A_2}$.
\end{lemma}



\subsection{Uniform bounds for some operators}

Let us show that some operators related with the semigroups are uniformly bounded, which we will use throughout the text.

\begin{lemma} \label{lem:bounded}
	The following operators are $L^2_W \to L^2_W$ bounded, uniformly on $t$, with norm depending on $n, \lambda$ and the 
doubling constant for $W$:  
	\begin{tasks}[label=(\roman*), label-offset=1em](4)
		\task \label{bounded:semigroup} $e^{-t^2L}$, 
		\task \label{bounded:L_semigroup} $t^2Le^{-t^2L}$, 
		\task \label{bounded:L2_semigroup} $t^4L^2e^{-t^2L}$, 
		\task \label{bounded:gradient_semigroup} $t\nabla e^{-t^2L}$, 
		\task \label{bounded:semigroup_adjoint} $e^{-t^2\L}$, 
		\task \label{bounded:L_semigroup_adjoint} $t^2\L e^{-t^2\L}$, 
		\task \label{bounded:L2_semigroup_adjoint} $t^4\L^2e^{-t^2\L}$, 
		\task \label{bounded:gradient_adjoint_semigroup} 
		$t\nabla e^{-t^2\L}$, 
		\task \label{bounded:semigroup_divergence} $t e^{-t^2L} \tdiv$, 
		\task \label{bounded:adjoint_semigroup_divergence} $te^{-t^2\L}\tdiv$, 
		\task \label{bounded:L_semigroup_divergence} $t^3 L e^{-t^2L} \tdiv$, 
		\task \label{bounded:adjoint_adjoint_semigroup_divergence} $t^3 \L e^{-t^2\L} \tdiv$. 
	\end{tasks}
\end{lemma}
\begin{remark}
	Some of the operators are in fact defined for functions $\mathbf{f} \in (L^2_W)^n$. In this case, we obviously mean that each of their components are bounded.
\end{remark}
\begin{proof}

\noindent \ref{bounded:semigroup} - \ref{bounded:L2_semigroup}.  Recall that $\mathcal{M}_W$ denotes the Hardy-Littlewood maximal operator with respect to the weighted measure $W(x) dx$.
If we let $T_t$ denote the operator
under consideration in \ref{bounded:semigroup}, \ref{bounded:L_semigroup} or \ref{bounded:L2_semigroup}, it suffices to observe 
that in each case we have the pointwise bound
\[ \sup_t |T_t f| \lesssim \mathcal{M}_Wf\,,
\]
using only the Gaussian bounds in Lemma 
\ref{lem:fund_sol_bound} (and Remark \ref{r2.6} 
in the case of \ref{bounded:L_semigroup} and \ref{bounded:L2_semigroup}), and the doubling property of $W$. We omit the routine details.

	
	To prove \ref{bounded:gradient_semigroup}, we set 
$u := e^{-t^2L} f$. Note that by the identity \eqref{eq:Luu}, we can \quotes{interpolate} the estimates in \ref{bounded:semigroup} and \ref{bounded:L_semigroup} using Cauchy-Schwarz as follows:
	\begin{multline*}
	\int_\Rn \abs{t \nabla u}^2 W dx
	\lesssim 
	\int_\Rn t^2 A \nabla u \cdot \nabla u W dx
	=
	\int_\Rn t^2 Lu u W dx
	\\ \leq 
	\left( \int_\Rn \abs{t^2 Lu}^2 W dx \right)^{1/2} \left( \int_\Rn \abs{u}^2 W dx \right)^{1/2} 
	\lesssim
	\int_\Rn \abs{f}^2 W dx,
	\end{multline*}
	which shows \ref{bounded:gradient_semigroup}.
	
	
	In turn, \ref{bounded:semigroup_adjoint}-\ref{bounded:L2_semigroup_adjoint} follow by duality from \ref{bounded:semigroup}-\ref{bounded:L2_semigroup}, while	\ref{bounded:gradient_adjoint_semigroup} follows \quotes{interpolating} \ref{bounded:semigroup_adjoint} and \ref{bounded:L_semigroup_adjoint} in the same way that we did it for \ref{bounded:gradient_semigroup}, this time using \eqref{eq:adjLuu} instead of \eqref{eq:Luu}. 
	
	
	Then, \ref{bounded:semigroup_divergence} and \ref{bounded:adjoint_semigroup_divergence} follow by 
duality from 
\ref{bounded:gradient_adjoint_semigroup}, and \ref{bounded:gradient_semigroup}, respectively.
	
	
	Lastly, \ref{bounded:L_semigroup_divergence} (resp. \ref{bounded:adjoint_adjoint_semigroup_divergence}) follows 
by first  dualizing, and then \quotes{interpolating} \ref{bounded:L_semigroup_adjoint} 
and \ref{bounded:L2_semigroup_adjoint} using \eqref{eq:adjLuu} 
(resp. \ref{bounded:L_semigroup} 
and \ref{bounded:L2_semigroup} using 
\eqref{eq:Luu}). 
	\end{proof}


\subsection{Off-diagonal estimates}


\begin{definition} \label{def:off_diagonal}
	We say that the operators \textbf{$S_t$ satisfy off-diagonal 
	estimates} (aka Gaffney estimates), 
	if there exist $c, C > 0$ (independent of $t$) such that it holds for every $t > 0$, and every pair of measurable sets $E$ and $F$,
	\begin{equation*}
	\int_F \abs{S_t f(x)}^2 W(x) dx
	\leq 
	C e^{-c\frac{\dist(E, F)^2}{t^2}} \int_E \abs{f(x)}^2 W(x) dx
	\qquad 
	\text{ if } \supp f \subset E\,.
	\end{equation*}
\end{definition}

\begin{lemma} \label{lem:off_diagonal}
	The following operators satisfy off-diagonal estimates, with constants depending only on $n, \lambda$ and 
the doubling constant of $W$:  
	\begin{tasks}[label=(\roman*), label-offset=1em](4)
		\task \label{off_diagonal:semigroup} $e^{-t^2L}$,
		\task \label{off_diagonal:L_semigroup} $t^2Le^{-t^2L}$,
		\task \label{off_diagonal:L2_semigroup} $t^4L^2e^{-t^2L}$,
		\task \label{off_diagonal:gradient_semigroup} $t\nabla e^{-t^2L}$,
		\task \label{off_diagonal:semigroup_adjoint} $e^{-t^2\L}$,
		\task \label{off_diagonal:L_semigroup_adjoint} $t^2\L e^{-t^2\L}$,
		\task \label{off_diagonal:L2_semigroup_adjoint} $t^4\L^2e^{-t^2\L}$,
		\task \label{off_diagonal:semigroup_divergence_adjoint} 
		$t e^{-t^2\L} \widetilde{\div}$,
		\task \label{off_diagonal:gradient_adjoint_semigroup} 
		$t\nabla e^{-t^2\L}$, 
		\task \label{off_diagonal:semigroup_divergence} $t e^{-t^2L} \tdiv$.
	\end{tasks}
\end{lemma}
\begin{proof}
	Fix sets $E, F \subset \Rn$. We may assume that 
	$d := \dist(E, F) > 5t$, as otherwise we may invoke Lemma~\ref{lem:bounded}. 
	
The bound for \ref{off_diagonal:semigroup} is a straightforward consequence of the pointwise bounds 
in Lemma \ref{lem:fund_sol_bound}.  We omit the routine details.
The off-diagonal estimates for \ref{off_diagonal:L_semigroup}-\ref{off_diagonal:L2_semigroup} follow in the same way as for \ref{off_diagonal:semigroup}, using Remark \ref{r2.6}.

	Let us now treat \ref{off_diagonal:gradient_semigroup}. We argue as in the proof of Caccioppoli's inequality. Let $u := e^{-t^2L} f$, with $f$ supported in $E$.  Choose $\psi \in \mathscr{C}^\infty(\R^n)$, where $0 \leq \psi \leq 1$ satisfies $\psi \equiv 1$ on $F$, $\dist(\supp \psi, E) \geq d/2$ (denoting $d := \dist(E, F)$), and $\norm{\psi}_\infty + d \norm{\nabla \psi}_\infty + d^2 \norm{\nabla^2 \psi}_\infty \leq C$. 
For future reference, we note that
\begin{equation}\label{eq2.15}
|\nabla \psi|^2 + |\nabla^2\psi| \lesssim d^{-2} \,\mathbf{1}_{\supp \psi}\,.
\end{equation}		
	
	Clearly, we have 
	\begin{equation} \label{eq:insert_cutoff}
	\int_F \abs{t\nabla u}^2 W dx 
	\leq 
	t^2 \int_\Rn \abs{\nabla u}^2 \psi^2 W dx.
	\end{equation}

	Compute, using \eqref{eq:Luu} and the symmetry of $A$:
	\begin{multline} \label{eq:cutoff_id_1}
	\int_\Rn A \nabla (u \psi) \cdot \nabla (u \psi) W dx
	=
	\int_\Rn L(u\psi) u\psi W 
	=
	- \sum_{i, j = 1}^n \int_\Rn a_{ij} D_i D_j (u \psi) u \psi W dx
	\\ 
	=
	\sum_{i, j = 1}^n  \left( - \int_\Rn a_{ij} D_i D_j u u \psi^2 W dx - 2 \int_\Rn a_{ij} D_i u D_j \psi u \psi W dx - \int_\Rn a_{ij} D_i D_j \psi u^2 \psi W dx \right)
	\\ 
	=
	\int_\Rn Lu u \psi^2 W dx - 2 \int_\Rn A \nabla u \cdot \nabla \psi u \psi W dx + \int_\Rn L \psi u^2 \psi W dx.
	\end{multline}
Also, 
	\begin{align} \label{eq:cutoff_id_2}
	\int_\Rn A \nabla (u \psi) \cdot \nabla (u \psi) W dx
	& \geq 
	\lambda \int_\Rn \abs{\nabla (u \psi)}^2 W dx
	\nonumber
	\\
	& =
	\lambda \left( \int_\Rn \abs{\nabla u}^2 \psi^2 W dx + 2 \int_\Rn \nabla u \cdot \nabla \psi u \psi W dx + \int_\Rn \abs{\nabla \psi}^2 u^2 W dx \right)
	\\
	& \geq 
	\lambda \int_\Rn \abs{\nabla u}^2 \psi^2 W dx + 2 \lambda \int_\Rn \nabla u \cdot \nabla \psi u \psi W dx.
	\nonumber
	\end{align}
Combining \eqref{eq:cutoff_id_1} and \eqref{eq:cutoff_id_2}, we may	dominate 
the right hand side of \eqref{eq:insert_cutoff} by
	\begin{multline} \label{eq:before_hiding}
	t^2 \int_\Rn \abs{\nabla u}^2 \psi^2 W dx\\[4pt]
	\lesssim \,
	t^2 \int_\Rn \abs{ \nabla u} \abs{ \nabla \psi} \abs{u \psi} W dx
	\,+\, t^2 \int_\Rn \abs{Lu u \psi^2} W 
	\,+\, t^2 \int_\Rn \abs{L \psi u^2 \psi} W\\[4pt]
	=: I + II + III\,.
	\end{multline}
For $\varepsilon > 0$ to be chosen momentarily, we have	
\[
I \lesssim \varepsilon t^2 \int_\Rn \abs{\nabla u}^2 \psi^2 W dx
	+ \varepsilon^{-1} \left(\frac{t}{d}\right)^2 \int_{\supp \psi} 
	u^2 W dx\,,
\]
where we have used ``Cauchy's inequality with $\varepsilon$" in term $I$, and then \eqref{eq2.15}.
Choosing $\varepsilon$ small enough, we may hide the first term, and 
use the off-diagonal bounds for \ref{off_diagonal:semigroup} to obtain the desired bound for the second term, since $t\leq d$.	
We estimate term $II$ using the Cauchy-Schwarz inequality, along with the off-diagonal bounds for \ref{off_diagonal:semigroup} and \ref{off_diagonal:L_semigroup}.

	The operators \ref{off_diagonal:semigroup_adjoint}-\ref{off_diagonal:semigroup_divergence_adjoint} are dual to the first four, and \ref{off_diagonal:semigroup_divergence} is dual to \ref{off_diagonal:gradient_adjoint_semigroup}.  It therefore remains to treat \ref{off_diagonal:gradient_adjoint_semigroup}.	To this end, we now set $u:= e^{-t^2\L} f$, with $f$ supported in $E$, so with $\psi$ as above, again using \eqref{eq:insert_cutoff}, we have
	\begin{multline*}
\int_F \abs{t\nabla u}^2 W dx \, \leq \int_{\Rn} \abs{t\nabla u}^2 \psi^2 W dx
	\lesssim  \,
	t^2 \int_\Rn a_{ij} D_j u D_i u \,\psi^2 W dx\\[4pt]
	=\, t^2 \int_\Rn  D_j u \,D_i (a_{ij} W u) \,\psi^2  dx
	\,-\, t^2 \int_\Rn u D_j u \,D_i (a_{ij}  W) \,\psi^2  dx \, =: \, \widetilde{I} + \widetilde{II}\, , 
\end{multline*}
where summation over $i, j$ is understood.
We first note that
	\begin{multline*}
\widetilde{I} \, = \, -t^2 \int_\Rn  u \,D_j D_i (a_{ij} W u) \,\psi^2  dx \,
-\,2 t^2 \int_\Rn   u \,D_i (a_{ij} W u) \,\psi D_j\psi  \,dx\, 
\\[4pt]
=\,
 t^2 \int_{\supp \psi}  u \L u \, \psi^2 W dx \,-\,
 2 t^2 \int_\Rn   u \,D_i (a_{ij} W u) \,\psi D_j\psi  \,dx
\,=: \widetilde{I}_1+
\widetilde{I}_2\,.
\end{multline*}
 We obtain the desired estimate 
 for $\widetilde{I}_1$ by 
 the off-diagonal bounds for \ref{off_diagonal:semigroup_adjoint} and \ref{off_diagonal:L_semigroup_adjoint} like for term $II$ in \eqref{eq:before_hiding}. Also, integrating by parts,
\[
\widetilde{I}_2 \,=\, 
2 t^2 \int_\Rn   u D_i u \, a_{ij}  \,\psi D_j\psi  \,W dx
\,+\,
2 t^2 \int_\Rn   u^2 \, a_{ij}  (D_i \psi D_j\psi + \psi D_iD_j \psi) \,W dx
\, =: \widetilde{I}_{2,1} + \widetilde{I}_{2,2}\,.
\]
We will cancel term $\widetilde{I}_{2,1} $ momentarily, but it may also be handled directly, exactly like term $I$ in \eqref{eq:before_hiding}, using Cauchy's inequality with $\varepsilon$, hiding the small term, and bounding the other term using the off-diagonal bounds for \ref{off_diagonal:semigroup_adjoint}. 
We estimate term $\widetilde{I}_{2,2}$ like terms $I$ and $III$ in 
\eqref{eq:before_hiding}, using \eqref{eq2.15}, the fact that $t\leq d$, and the off-diagonal bounds for \ref{off_diagonal:semigroup_adjoint}. 

Since $W$ is an adjoint solution, 
\begin{multline*}
\widetilde{II}  \,=\, t^2 \int u^2\, \psi D_j \psi \, D_i (a_{ij} W) dx \\[4pt]
=\, -2t^2 \int u D_i u\, \psi D_j \psi \, a_{ij} W dx\, -\,
t^2 \int u^2\, (D_i\psi D_j \psi  + \psi D_iD_j \psi)\, a_{ij} W dx\,=:\,
\widetilde{II}_1 + \widetilde{II}_2\,.
\end{multline*}
Observe that $\widetilde{II}_1\equiv -\widetilde{I}_{2,1}$,
and $\widetilde{II}_2 \equiv -\frac12 \widetilde{I}_{2,2} $. 
\end{proof}

\begin{lemma} \label{lem:off_diagonal_P_t_product}
	Let $K_t$ be a convolution type operator as in Lemma~\ref{lem:P_t_bounded}, with a {\tt compactly supported kernel}. If 
$h=h(x,t)$ is a function such that $f\mapsto h(\cdot,t) K_t f$ is bounded on $L^2_W$, uniformly in $t$ (i.e. $\norm{h(\cdot,t) K_t}_{L^2_W \to L^2_W} \lesssim 1$ uniformly in $t$), then $h(\cdot,t) K_t$ satisfies off-diagonal estimates.
\end{lemma}

We omit the trivial proof.

\begin{lemma} \label{lem:off_diagonal_P_t_composition}
	Let $\{U_t\}_{t>0}$ and $\{U'_t\}_{t>0}$ be two families of operators, each  satisfying off-diagonal estimates, then the composition $U_t U'_t$ also
	satisfies off-diagonal estimates for each $t$.
\end{lemma}
We omit the routine proof.


\subsection{Estimates for differences and gradients}

\begin{lemma} \label{lem:semigroup_difference_estimate}
	For $f \in \Lip$ and $t \leq \ell(Q)$, we have 
	\begin{equation*}
	\int_Q \abs{e^{-t^2L}f(x) - f(x)}^2 W(x) dx 
	\lesssim 
	t^2 \norm{\nabla f}_\infty^2 W(Q)
	\end{equation*}
	and 
	\begin{equation*}
	\int_Q \abs{e^{-t^2\L}f(x) - f(x)}^2 W(x) dx 
	\lesssim 
	t^2 \norm{\nabla f}_\infty^2 W(Q),
	\end{equation*}
	where the implicit constants depend on $n, \lambda$ and $[W]_{A_2}$.
\end{lemma}
\begin{proof}
	Let us show the first estimate. 
	Cover $Q$ by non-overlapping cubes $Q_k$ with sidelength $t/2 < \ell(Q_k) \leq t$. Note  that $e^{-t^2L} \mathbf{1} = \mathbf{1}$, since $L\mathbf{1}=0$.
Letting $[f]_E$ denote the average of $f$ over the set $E$,
	\begin{equation}\label{eq2.23}
	\norm{e^{-t^2L} f - f}_{L^2_W(Q)}^2
	=
	\sum_k \norm{e^{-t^2L} (f - [f]_{2Q_k}) - (f - [f]_{2Q_k})}_{L^2_W(Q_k)}^2
	=:
	\sum_k A_k^2
	\end{equation}
	We then have
	\begin{multline}\label{eq2.24}
	A_k 
	\leq 
	\norm{e^{-t^2L} \left( (f - [f]_{2Q_k}) \mathbf{1}_{2Q_k} \right) }_{L^2_W(Q_k)}
	\,+\, \norm{f - [f]_{2Q_k}}_{L^2_W(Q_k)}
	\\ + \,\sum_{j=1}^\infty \norm{e^{-t^2L} \left( (f - [f]_{2Q_k}) \mathbf{1}_{2^{j+1}Q_k \setminus 2^jQ_k} \right) }_{L^2_W(Q_k)}
	=:
	I^{(k)} + \sum_{j=1}^\infty II^{(k)}_j. 
	\end{multline}
	
	Using the boundedness of $e^{-t^2L}$ from Lemma~\ref{lem:bounded} and the weighted version of Poincaré's inequality \cite[15.26]{HKM}, 
	we deduce 
	\begin{equation*}
	I^{(k)} 
	\lesssim 
	\norm{f - [f]_{2Q_k}}_{L^2_W(2Q_k)}
	\lesssim 
	\ell(Q_k) \norm{\nabla f}_{L^2_W(2Q_k)}
	\leq 
	t \norm{\nabla f}_{L^2_W(2Q_k)}
	\lesssim
	t \norm{\nabla f}_\infty \sqrt{W(Q_k)}.
	\end{equation*}
	
	For convenience of notation in the rest of this argument, we replace the constant $c$ by $4c$ in the off-diagonal estimates for $e^{-t^2L}$ from Lemma~\ref{lem:off_diagonal}.  Thus,
	\begin{equation*}
	\begin{aligned}
	II^{(k)}_j
	& \lesssim 
	e^{-c 4^{j+1} \frac{\ell(Q_k)^2}{t^2}} \norm{f - [f]_{2Q_k}}_{L^2_W(2^{j+1}Q_k)}
	\\ 
	& \leq 
	e^{-c 4^{j}} \norm{f - [f]_{2^{j+1}Q_k}}_{L^2_W(2^{j+1}Q_k)}
	+ \sum_{i=1}^j e^{-c 4^j} \norm{[f]_{2^{i+1}Q_k} - [f]_{2^iQ_k}}_{L^2_W(2^{j+1}Q_k)}
	\\ 
	& 
	\lesssim
	e^{-c 4^{j}} \norm{f - [f]_{2^{j+1}Q_k}}_{L^2_W(2^{j+1}Q_k)}
	+ \sum_{i=1}^j e^{-c 4^j} C_D^{(j-i)/2} \norm{f - [f]_{2^{i+1}Q_k}}_{L^2_W(2^{i+1}Q_k)}\,.
	\end{aligned}
	\end{equation*}
Hence, summing on $j$ and interchanging the order of summation we obtain
	\begin{equation*}
	\sum_{j=1}^\infty \sum_{i=1}^j e^{-c 4^j} C_D^{(j-i)/2} \norm{f - [f]_{2^{i+1}Q_k}}_{L^2_W(2^{i+1}Q_k)}
	\lesssim 
	\sum_{i=1}^\infty e^{-\frac{c}{2} 4^i} \norm{f - [f]_{2^{i+1}Q_k}}_{L^2_W(2^{i+1}Q_k)},
	\end{equation*}
	and therefore we have, by Poincaré's inequality,
	\begin{multline*}
	\sum_{j=1}^\infty II^{(k)}_j
	\lesssim 
	\sum_{j=1}^\infty e^{-\frac{c}{2} 4^j} \norm{f - [f]_{2^{j+1}Q_k}}_{L^2_W(2^{j+1}Q_k)}
	\lesssim 
	\sum_{j=1}^\infty e^{-\frac{c}{2} 4^j} 2^j \ell(Q_k) \norm{\nabla f}_{L^2_W(2^{j+1}Q_k)}
	\\ \lesssim 
	t \sum_{j=1}^\infty e^{-\frac{c}{4} 4^j} \norm{\nabla f}_\infty \sqrt{W(2^{j+1}Q_k)}
	\lesssim 
	t \norm{\nabla f}_\infty \sum_{j=1}^\infty e^{-\frac{c}{5} 4^j}  \sqrt{W(Q_k)} 
	\lesssim 
	t \norm{\nabla f}_\infty \sqrt{W(Q_k)}.
	\end{multline*}	
	Thus, going back to $A_k$ (see \eqref{eq2.23}-\eqref{eq2.24}) we obtain 
	\begin{equation*}
	A_k
	\lesssim 
	t \norm{\nabla f}_\infty \sqrt{W(Q_k)},
	\end{equation*}
	and therefore, since the cubes $Q_k$ are non-overlapping,
	\begin{equation*}
	\norm{e^{-t^2L} f - f}_{L^2_W(Q)}^2
	\lesssim 
	t^2 \norm{\nabla f}_\infty^2 \sum_k W(Q_k)
	=
	t^2 \norm{\nabla f}_\infty^2 W(Q).
	\end{equation*}
	
	The corresponding estimate for $e^{-t^2\L}$ holds with exactly the same proof, for the operator is also bounded and satisfies 
off-diagonal estimates (see Lemmas~\ref{lem:bounded} and \ref{lem:off_diagonal}), and $e^{-t^2\L} \mathbf{1} = \mathbf{1}$ since 
 $\L\mathbf{1} = \frac{1}{W} L^*W = 0$.
\end{proof}


\begin{lemma} \label{lem:gradient_semigroup_difference_estimate}
	For $f \in \Lip$ and $t \leq \ell(Q)$, we have 
	\begin{equation*}
	\int_Q \abs{\nabla e^{-t^2L}f(x) }^2 W(x) dx
	\lesssim
	\norm{\nabla f}_\infty^2 W(Q)
	\end{equation*}
	and 
	\begin{equation*}
	\int_Q \abs{\nabla e^{-t^2\L}f(x)}^2 W(x) dx 
	\lesssim  
	\norm{\nabla f}_\infty^2 W(Q),
	\end{equation*}
	where the implicit constant depends on $n, \lambda$ and $[W]_{A_2}$.
\end{lemma}
\begin{proof}
	The reader can check that the proof is the same as for Lemma~\ref{lem:semigroup_difference_estimate}, this time using the operators $t\nabla e^{-t^2L}$ and $t\nabla e^{-t^2\L}$. Indeed, the proof of the preceeding Lemma used only
	the boundedness of the operator (Lemma~\ref{lem:bounded}), the off-diagonal estimates (Lemma~\ref{lem:off_diagonal}), and the conservation property which allows us to subtract constants. In this case, 
	\begin{equation*}
	\norm{\nabla e^{-t^2L} f }_{L^2_W(Q)}^2
	=
	\sum_k \norm{\nabla  e^{-t^2L} (f - [f]_{2Q_k}) }_{L^2_W(Q_k)}^2\,,
	\end{equation*}
	and similarly for $\L$,
	because $e^{-t^2L} \mathbf{1} = \mathbf{1} =e^{-t^2\L}\mathbf{1}$, as before. After this, the proof 
	is the same as that of Lemma \ref{lem:semigroup_difference_estimate}. 
\end{proof}


\section{The Kato problem for $L$} \label{sec:kato_L}

In this section our goal is to prove the following result:
\begin{theorem} \label{th:kato_L}
	It holds
	\begin{equation*}
	\norm{\sqrt{L} f}_{L^2_W} \lesssim \norm{\nabla f}_{L^2_W},
	\end{equation*}
	where the hidden constant depends on $n, \lambda$ and $[W]_{A_2}$.
\end{theorem}

As noted above (see Remark \ref{remark:L-easy}), in the case that
the coefficient matrix has small enough BMO norm, we could deduce Theorem \ref{th:kato_L} as an easy consequence of certain known results.  Instead, 
following the easier part of our proof of Theorem \ref{th:kato_adj} below,
we shall give a self-contained, direct argument, which does not rely on an explicit assumption of smallness in BMO, but only on the validity of \eqref{eq:Luu}, and the assumption that $W\in A_2$.

To prove the theorem, let us use the representation of the square root operator via the Functional Calculus formula
\begin{equation*}
\sqrt{L} f = a \int_0^\infty t^3L^2e^{-2t^2L}f \frac{dt}{t},
\end{equation*}
where $a = (\int_0^\infty t^3e^{-2t^2} \frac{dt}{t})^{-1} = \sqrt{\frac{128}{\pi}}$ is just a normalizing constant, to estimate, using duality and later Cauchy-Schwarz,
\begin{equation*} 
\abs{\pair{\sqrt{L}f}{g}_{L^2_W}}^2
\leq 
a^2 \left( \int_0^\infty \norm{tLe^{-t^2L}f(x)}^2_{L^2_W} \frac{dt}{t} \right) \left( \int_0^\infty \norm{t^2\L e^{-t^2\L}g(x)}^2_{L^2_W} \frac{dt}{t} \right).
\end{equation*}
With this decomposition, we will finish the proof of Theorem~\ref{th:kato_L} by duality once we prove the following Lemma~\ref{lem:sq_fun_L_easy} and Lemma~\ref{lem:sq_fun_L}.

The desired bound for the second factor is the following.


\begin{lemma} \label{lem:sq_fun_L_easy}
	It holds 
	\begin{equation} 
	\int_0^\infty \int_{\R^n} \abs{t^2 \L e^{-t^2\L}g(x)}^2 W(x) dx \frac{dt}{t} \lesssim \norm{g}_{L^2_W}^2,
	\end{equation}
	where the implicit constant depends only on 
	$n, \lambda$ and 
the doubling constant for $W$.	
\end{lemma}
We could obtain the conclusion of the lemma by invoking the abstract McIntosh and Yagi theorem \cite{Y}, \cite{Mc}, but instead,
we will give a self-contained and more elementary proof, using quasi-orthogonality arguments.

\begin{proof}
	We abbreviate $V_t := t^2 L e^{-t^2L}$. Its adjoint within $L^2_W$ is $\V_t := t^2 \L e^{-t^2\L}$. To make the argument rigorous, 
	 given a small positive $\varepsilon$, we set
$V_t \equiv 0 \equiv \V_t$ whenever $t\leq \varepsilon$ or $t\geq 1/\varepsilon$,
and we obtain quantitative bounds that are uniform in $\varepsilon$.
We compute, using  duality, Fubini and Cauchy-Schwarz,
	\begin{multline} \label{eq:easy_sq_fun_duality}
	\int_0^\infty \int_\Rn \abs{\V_t g(x)}^2 W(x) dx \frac{dt}{t}
	=
	\int_0^\infty \int_\Rn \V_t g(x) \V_t g(x) W(x) dx \frac{dt}{t}
	\\ =
	\int_0^\infty \int_\Rn g(x) V_t \V_t g(x) W(x) dx \frac{dt}{t}
	=
	\int_\Rn \int_0^\infty V_t \V_t g(x) \frac{dt}{t} g(x) W(x) dx
	\\ \leq 
	\norm{g}_{L^2_W} \norm{\int_0^\infty V_t \V_t g \frac{dt}{t}}_{L^2_W}.
	\end{multline}

	To deal with the second term, let us first establish a useful fact:
	\begin{claim}
		It holds, for any $t, s > 0$,
		\begin{equation*}
		\norm{\V_t V_s}_{L^2_W \to L^2_W} \lesssim \min \left\{ \frac{s}{t}, \frac{t}{s} \right\}.
		\end{equation*}
	\end{claim}
	\begin{proof}[Proof of the claim] We may assume that 
	$\varepsilon < s,t <1/\varepsilon$.
		If $s \leq t$, we can compute using \eqref{eq:Lu+adjLu} and the uniform bounds from Lemma~\ref{lem:bounded}, 
		\begin{multline*}
		\norm{\V_t V_s} 
		= 
		\norm{t^2 s^2 \L e^{-t^2\L} L e^{-s^2L}}
		\lesssim 
		\norm{t^2 s^2 \L e^{-t^2\L} \L e^{-s^2L}} + \norm{t^2 s^2 \L e^{-t^2\L} \div A \nabla e^{-s^2L}}
		\\ \leq 
		\frac{s^2}{t^2} \norm{t^4 \L^2 e^{-t^2\L}} \norm{e^{-s^2L}} + \frac{s}{t} \norm{t^3 \L e^{-t^2\L} \div} \norm{A} \norm{s \nabla e^{-s^2L}} 
		\lesssim 
		\frac{s^2}{t^2} + \frac{s}{t}
		\lesssim 
		\frac{s}{t}.
		\end{multline*}
		
		In a similar fashion we can compute, when $s > t$, using again \eqref{eq:Lu+adjLu} and Lemma~\ref{lem:bounded},
		\begin{multline*}
		\norm{\V_t V_s} 
		= 
		\norm{t^2 s^2 e^{-t^2\L} \L L e^{-s^2L}}
		\lesssim 
		\norm{t^2 s^2 e^{-t^2\L} L L e^{-s^2L}} + \norm{t^2 s^2 e^{-t^2\L} \div A \nabla L e^{-s^2L}}
		\\ \leq 
		\frac{t^2}{s^2} \norm{e^{-t^2\L}} \norm{s^4 L^2 e^{-s^2L}} + \frac{t}{s} \norm{t e^{-t^2\L} \div} \norm{A} \norm{s^3 \nabla L e^{-s^2L}}
		\lesssim 
		\frac{t^2}{s^2} + \frac{t}{s}
		\lesssim 
		\frac{t}{s}.
		\end{multline*}
	\end{proof}
	
	With this almost-orthogonality result, we can estimate the last term in \eqref{eq:easy_sq_fun_duality} as follows:
	\begin{align*}
	\norm{\int_0^\infty V_t \V_t g \frac{dt}{t}}_{L^2_W}^2
	& =
	\int_\Rn \left(\int_0^\infty V_t \V_t g(x) \frac{dt}{t}\right) \left(\int_0^\infty V_s \V_s g(x) \frac{ds}{s}\right) W(x) dx
	\\ & =
	\int_0^\infty \int_0^\infty \int_\Rn V_t \V_t g(x) V_s \V_s g(x) W(x) dx \frac{dt}{t} \frac{ds}{s}
	\\ & =
	\int_0^\infty \int_0^\infty \int_\Rn \V_t g(x) \V_t V_s \V_s g(x) W(x) dx \frac{dt}{t} \frac{ds}{s}
	\\ & \leq 
	\int_0^\infty \int_0^\infty \left( \int_\Rn \abs{\V_t g}^2 W dx \right)^{1/2} \left( \int_\Rn \abs{\V_t V_s \V_s g}^2 Wdx \right)^{1/2} \frac{dt}{t} \frac{ds}{s}
	\\ & \leq 
	\int_0^\infty \int_0^\infty \min \left\{ \frac{s}{t}, \frac{t}{s} \right\} \left( \int_\Rn \abs{\V_t g}^2 W dx \right)^{1/2} \left( \int_\Rn \abs{\V_s g}^2 Wdx \right)^{1/2} \frac{dt}{t} \frac{ds}{s}
	\\ & \lesssim 
	\int_0^\infty \int_0^\infty \min \left\{ \frac{s}{t}, \frac{t}{s} \right\} \int_\Rn \abs{\V_t g(x)}^2 W(x) dx  \frac{dt}{t} \frac{ds}{s}
	\\ & \lesssim 
	\int_0^\infty \int_\Rn \abs{\V_t g(x)}^2 W(x) dx \frac{dt}{t},
	\end{align*}
	Plugging this estimate into \eqref{eq:easy_sq_fun_duality}, we have
	\begin{equation*}
	\int_0^\infty \int_\Rn \abs{\V_t g(x)}^2 W(x) dx \frac{dt}{t}
	\lesssim
	\norm{g}_{L^2_W} \left(\int_0^\infty \int_\Rn \abs{\V_t g(x)}^2 W(x) dx \frac{dt}{t} \right)^{1/2}, 
	\end{equation*}
	from which the result readily follows (recall that we have effectively truncated so $\varepsilon < t <1/\varepsilon$, hence the integrals are finite).
\end{proof}


Let us turn our attention to the other square function estimate.

\begin{lemma} \label{lem:sq_fun_L}
	It holds 
	\begin{equation} \label{eq:kato_L}
	\int_0^\infty \int_{\R^n} \abs{tLe^{-t^2L}f(x)}^2 W(x) dx \frac{dt}{t} \lesssim \norm{\nabla f}_{L^2_W}^2,
	\end{equation}
	where the implicit constant depends on $n, \lambda$ and $[W]_{A_2}$.
\end{lemma}

We will devote the rest of the section to the proof of Lemma~\ref{lem:sq_fun_L}. As above, set $V_t := t^2Le^{-t^2L}$ and decompose, with the help of an approximate identity $P_t$, 
\begin{equation} \label{eq:decomposition_L}
tLe^{-t^2L} = t^{-1}V_t(I-P_t) + t^{-1}V_tP_t =: R_t + T_t.
\end{equation}

The proof of Lemma~\ref{lem:sq_fun_L}, and hence of Theorem~\ref{th:kato_L}, will come immediately from the next two lemmas.


\begin{lemma} \label{lem:R_t}
	With the notations of \eqref{eq:decomposition_L}, we have
	\begin{equation*}
	\int_0^\infty \norm{R_t f}_{L^2_W}^2 \frac{dt}{t} \lesssim \norm{\nabla f}_{L^2_W}^2,
	\end{equation*}
	where the implicit constant depends on $n, \lambda$ and $[W]_{A_2}$.
\end{lemma}
\begin{proof}
	Choose the approximate identity $P_t := e^{-t^2(-\Delta)}$. With it, we can compute, using the Fundamental Theorem of Calculus,
	\begin{equation*}
	\begin{aligned}
	R_t f 
	& =
	t^{-1}V_t (I-P_t) f
	=
	 V_t \left( \frac{1}{t} (I-P_t) f \right) 
	=
	 V_t \left( - \frac{1}{t} \int_0^t \frac{\partial}{\partial s} P_s f ds \right)
	\\ 
	& =
	-2  V_t \left( \frac{1}{t} \int_0^t s P_s \,\Delta f ds \right)
	=
	-2  V_t \left( \frac{1}{t} \int_0^t s P_s \div \nabla f ds \right)
	=:
	2  V_t \left( \frac{1}{t} \int_0^t \vec{Q}_s \nabla f ds \right).
	\end{aligned}
	\end{equation*}
Now, using the boundedness on $L^2_W$ of $V_t = t^2Le^{-t^2L}$ (see Lemma~\ref{lem:bounded}), Hardy's inequality and the fact that $\vec{Q}_s$ satisfies the square function estimate of Lemma~\ref{lem:Q_s_square_fun}
(see Remark~\ref{rem:derivative_P_t}), 
we obtain the desired estimate:
	\begin{multline*}
	\int_0^\infty \norm{R_t f}_{L^2_W}^2 \frac{dt}{t}
	=
	\int_0^\infty  \int_{\R^n} \abs{ 2  V_t \left( \frac{1}{t} \int_0^t \vec{Q}_s \nabla f ds \right)(x) }^2 W(x) dx \frac{dt}{t}
	\\ 
	\lesssim 
	\int_0^\infty  \int_{\R^n} \abs{ \frac{1}{t} \int_0^t \vec{Q}_s \nabla f(x) ds}^2 W(x) dx \frac{dt}{t}
	\lesssim 
	\int_0^\infty  \int_{\R^n} \abs{ \vec{Q}_t \nabla f(x) }^2 W(x) dx \frac{dt}{t}
	\lesssim 
	\norm{\nabla f}_{L^2_W}^2.
	\end{multline*}
\end{proof}


\begin{lemma} \label{lem:S_t}
	With the notations of \eqref{eq:decomposition_L}, we have
	\begin{equation*}
	\int_0^\infty \norm{T_t f}_{L^2_W}^2 \frac{dt}{t} \lesssim \norm{\nabla f}_{L^2_W}^2,
	\end{equation*}
	where the hidden constant depends on $n, \lambda$ and $[W]_{A_2}$.
\end{lemma}
\begin{proof}
	Simply compute, using the boundedness of $e^{-t^2L}$ from Lemma~\ref{lem:bounded},
	 and the square function bounds of Lemma~\ref{lem:Q_s_square_fun} 
	 (see Remark~\ref{rem:derivative_P_t}),
	\begin{multline*}
	\int_0^\infty \norm{T_t f}_{L^2_W}^2 \frac{dt}{t}  
	\lesssim 
	\int_0^\infty \norm{ t L P_t f }_{L^2_W}^2 \frac{dt}{t}
	\lesssim 
	\sum_{i, j = 1}^n \int_0^\infty \norm{ t D_i D_j P_t f }_{L^2_W}^2 \frac{dt}{t}
	\\ =
	\sum_{i, j = 1}^n \int_0^\infty \norm{ t D_i P_t D_j f }_{L^2_W}^2 \frac{dt}{t}
	=:
	\sum_{i, j = 1}^n \int_0^\infty \norm{ Q_t^{(i)} D_j f }_{L^2_W}^2 \frac{dt}{t}
	\\ \lesssim 
	\sum_{j = 1}^n \norm{D_j f}_{L^2_W}^2
	=
	\norm{\nabla f}_{L^2_W}^2.
	\end{multline*}
\end{proof}


\section{The Kato problem for $\L$} \label{sec:kato_adj}

In this section our goal is to prove the following result, which is really the main result in this paper:
\begin{theorem} \label{th:kato_adj}
	It holds
	\begin{equation*}
	\norm{\sqrt{\L} f}_{L^2_W} \lesssim \norm{\nabla f}_{L^2_W},
	\end{equation*}
	where the implicit constant depends on $n, \lambda$ and $[W]_{A_2}$.
\end{theorem}


\subsection{Reduction to a quadratic estimate}

To prove Theorem~\ref{th:kato_adj}, let us again use the representation of the square root operator via the formula
\begin{equation*}
\sqrt{\L} f = a \int_0^\infty t^3\L^2e^{-2t^2\L}f \frac{dt}{t},
\end{equation*}
so that
\begin{equation*}
\abs{\pair{\sqrt{\L}f}{g}_{L^2_W}}^2
\leq 
a^2 \left( \int_0^\infty \norm{t\L e^{-t^2\L}f(x)}^2_{L^2_W} \frac{dt}{t} \right) \left( \int_0^\infty \norm{t^2L e^{-t^2L}g(x)}^2_{L^2_W} \frac{dt}{t} \right).
\end{equation*}
Theorem~\ref{th:kato_adj} then follows immediately from Lemma~\ref{lem:sq_fun_adj_easy} and Theorem~\ref{th:sq_fun_adj} below.


We estimate the second square function via the following lemma. 
\begin{lemma} \label{lem:sq_fun_adj_easy}
	It holds 
	\begin{equation} 
	\int_0^\infty \int_{\R^n} \abs{t^2 L e^{-t^2L}g(x)}^2 W(x) dx \frac{dt}{t} \lesssim \norm{g}_{L^2_W}^2,
	\end{equation}
	where the implicit constant 
	depends on $n, \lambda$ and $[W]_{A_\infty}$.
\end{lemma}

\begin{proof}
The lemma can be proved either by invoking the McIntosh and Yagi theorem \cite{Y}, \cite{Mc}, or via a self-contained elementary proof using quasi-orthogonality.  For the latter path, the proof follows that of Lemma~\ref{lem:sq_fun_L_easy} {\em mutatis mutandis}, simply reversing the roles of $V_t$ and $\V_t$.   We omit the details.
\end{proof}


Let us now turn the attention to the other square function estimate, which is in fact the core of this paper.

\begin{theorem} \label{th:sq_fun_adj}
	It holds 
	\begin{equation} \label{eq:kato_adj}
	\int_0^\infty \int_{\R^n} \abs{t\L e^{-t^2\L}f(x)}^2 W(x) dx \frac{dt}{t} \lesssim \norm{\nabla f}_{L^2_W}^2,
	\end{equation}
	where the implicit constant depends on $n, \lambda$ and $[W]_{A_2}$.
\end{theorem}

The rest of this section is devoted to the proof of Theorem~\ref{th:sq_fun_adj}. We start by splitting 
\begin{equation} \label{eq:decomposition_adj}
t\L e^{-t^2\L} = t\L e^{-t^2\L}(I-P_t^2) + t\L e^{-t^2\L}P_t^2 =: \widetilde{R}_t + \widetilde{T}_t,
\end{equation}
where $P_t$ is a nice approximate identity with a smooth compactly supported
convolution kernel $\varphi_t(x)= t^{-n} \varphi(x/t)$, which we take to be even.  For future reference, let us record the following well-known observation:
\[
t \partial_t \big(\widehat{\varphi}(t\xi)\big)^2 = 
2 (\nabla\widehat{\varphi})(t\xi) \cdot t\xi \widehat{\varphi}(t\xi)
= \,c \,\widehat{(x\varphi(x))}(t\xi) \cdot \widehat{(\nabla \varphi)}(t\xi)
=: c \, \widehat{\psi^{(1)}}(t\xi) \widehat{\psi^{(2)}}(t\xi) \,,
\]
where $c$ is a harmless constant, and $\psi^{(1)} (x):= x\varphi(x)$ and $\psi^{(2)}:= \nabla \varphi$ are both $\mathscr{C}_c^\infty$ functions with mean value zero (here we are using that $\varphi$ is even, in the case of $\psi^{(1)}$).  Hence,
\begin{equation}\label{Qsdef}
t \partial_t P_t^2 = \,c\, Q_t^{(1)} Q_t^{(2)}\,,
\end{equation}
where $Q_t^{(k)}$ is the convolution kernel with kernel 
$\psi_t^{(k)}(x):=t^{-n}\psi^{(k)}(x/t)$, $k=1,2$,
and therefore each of $Q_t^{(1)}$,  $Q_t^{(2)}$ satisfies the square function bound of Lemma \ref{lem:Q_s_square_fun} (and each is bounded on
$L^2_W$, uniformly in $t$).


\begin{lemma}\label{lemma4.8}
	With the notations of \eqref{eq:decomposition_adj}, we have
	\begin{equation}\label{eqrttilde}
	\int_0^\infty \norm{\widetilde{R}_t f}_{L^2_W}^2 \frac{dt}{t} \lesssim \norm{\nabla f}_{L^2_W}^2,
	\end{equation}
	where the implicit constant depends on $n, \lambda$ and $[W]_{A_2}$.
\end{lemma}
\begin{proof}
	Using the preceeding observations, we may follow the 
	proof of Lemma~\ref{lem:R_t},  invoking 
Lemma~\ref{lem:bounded} to obtain that $t^2 \L e^{-t^2\L}$ is $L^2_W$ bounded, to obtain \eqref{eqrttilde}.
\end{proof}


Applying now \eqref{eq:Lu+adjLu} to $u = P_t^2 f$ we obtain 
\begin{equation}\label{Ttsplit}
\widetilde{T}_t f
=
t e^{-t^2\L} \L P_t^2 f
=
- t e^{-t^2\L} L P_t^2 f - 2 t e^{-t^2\L} \widetilde{\div} (A \nabla (P_t^2 f)).
\end{equation}


\begin{lemma}\label{lemma4.10}
	We have 
	\begin{equation*}
	\int_0^\infty \norm{ t e^{-t^2\L} L P_t^2 f }^2_{L^2_W} \frac{dt}{t} 
	\lesssim 
	\norm{\nabla f}_{L^2_W}^2,
	\end{equation*}
	where the implicit constant depends on $n, \lambda$ and $[W]_{A_2}$.
\end{lemma}
\begin{proof}
	The proof is the same as that in Lemma~\ref{lem:S_t} once we use that $e^{-t^2\L} : L^2_W \to L^2_W$ is uniformly bounded by Lemma~\ref{lem:bounded}, and that $P_t$ are uniformly bounded on $L^2_W$ by Lemma~\ref{lem:P_t_bounded}.
\end{proof}


Therefore, to finish the proof of Theorem~\ref{th:sq_fun_adj} (and hence of Theorem~\ref{th:kato_adj}), it remains to show 
\begin{equation} \label{goal:sq_fun}
\int_0^\infty \norm{ t e^{-t^2\L} \widetilde{\div} (A \nabla (P_t^2 f)) }^2_{L^2_W} \frac{dt}{t} 
\lesssim 
\norm{\nabla f}_{L^2_W}^2.
\end{equation}


\subsection{Reduction to a Carleson measure estimate}

For ${\bf g} = (g_1,g_2,...,g_n)$, write 
\begin{equation}\label{thetadef}
\theta_t \mathbf{g} 
:=
t e^{-t^2\L} \widetilde{\div} (A \mathbf{g})
\left(
=
t e^{-t^2\L} \frac{1}{W} \div (W A \mathbf{g})
\right).
\end{equation}
With this notation, the remaining estimate \eqref{goal:sq_fun} 
becomes 
\begin{equation} \label{goal:sq_fun_theta}
\int_0^\infty \norm{ \theta_t \nabla (P_t^2 f) }^2_{L^2_W} \frac{dt}{t} 
\lesssim 
\norm{\nabla f}_{L^2_W}^2.
\end{equation}

Let us also define the operator 
\begin{equation*}
\ttheta_t \mathbf{g} 
:=
t e^{-t^2\L} \left( \widetilde{\div} (A \mathbf{g}) - \frac12 \sum_{i, j = 1}^n a_{ij} D_i \mathbf{g}_j \right),
\end{equation*}
so that, taking ${\bf g} = \nabla u$, and using \eqref{eq:Lu+adjLu},
\begin{equation}\label{thetaidentity}
\ttheta_t \nabla  u
=
-\frac12 te^{-t^2\L} \L u.
\end{equation}

It will be convenient to use both operators at different stages of the proof. Note that trivially, 
$\ttheta_t {\bf e} = \theta_t {\bf e}$, for any constant vector ${\bf e}$.
In particular, if  $\mathbb{1}$ denotes the $n\times n$ identity matrix, then
\begin{equation}\label{eq4.12}
\ttheta_t \mathbb{1} = \theta_t \mathbb{1}\,, 
\end{equation}
where we naturally define
$\ttheta_t \mathbb{1}= \theta_t \mathbb{1}$ as a vector-valued function
whose $k^{th}$ entry is $\ttheta_t {\bf e}^k = \theta_t {\bf e}^k$, with 
${\bf e}^k $ equal to the standard unit basis vector in the $x_k$ direction.




To prove \eqref{goal:sq_fun}, as in the divergence form case treated in \cite{AHLMcT}, we begin with a ``$T1$" reduction.

\begin{lemma} \label{lem:difference_term_LP}
	We have
	\begin{equation*}
	\int_0^\infty \norm{ 
\theta_t P_t^2 \mathbf{g} - (\theta_t \mathbb{1}) \cdot (P_t^2 \mathbf{g}) 
	}^2_{L^2_W} \frac{dt}{t} 
	\lesssim 
	\norm{\mathbf g}_{L^2_W}^2,
	\end{equation*}
	where the implicit constant depends on $n, \lambda$ and $[W]_{A_2}$.
\end{lemma}
\begin{proof}
	Write $U_t \mathbf{g} := \theta_t P_t^2 \mathbf{g} - (\theta_t \mathbb{1}) \cdot (P_t^2 \mathbf{g})$. By 
Lemma~\ref{lem:orthogonality}, it suffices to show that 
	\begin{equation*}
	\norm{U_t}_{L^2_W \to L^2_W} \lesssim 1
	\end{equation*}
	uniformly on $t$, and for some $\alpha > 0$, and for any nice operator $Q_s$ as in Lemma \ref{lem:Q_s_square_fun},  with a compactly supported kernel,
	\begin{equation*}
	\norm{U_t Q_s}_{L^2_W \to L^2_W} \lesssim \min \left\{ \frac{s}{t}, \frac{t}{s} \right\}^\alpha.
	\end{equation*}
	These two estimates, and hence the conclusion of Lemma \ref{lem:difference_term_LP}, will follow at once from the next claims and Lemma~\ref{lem:P_t_bounded}. 
	
	
	\begin{claim} \label{claim:U_t_bounded_1}
		We have, uniformly on $t$, 
		\begin{equation*}
		\norm{\theta_t P_t}_{L^2_W \to L^2_W} \lesssim 1,
		\end{equation*}
		where the implicit constant depends on $n, \lambda$ and $[W]_{A_2}$.
	\end{claim}
	\begin{proof}[Proof of the claim]
		Just compute, with the aid of Lemmas~\ref{lem:P_t_bounded} and \ref{lem:bounded},
		\begin{equation*}
		\norm{\theta_t P_t}_{L^2_W \to L^2_W}
		\leq 
		\norm{t e^{-t^2\L} \widetilde{\div}}_{L^2_W \to L^2_W} \norm{A}_\infty \norm{P_t}_{L^2_W \to L^2_W}
		\lesssim
		1.
		\end{equation*}
	\end{proof}
	
	
	\begin{claim} \label{claim:U_t_bounded_2}
		We have, uniformly on $t$, 
		\begin{equation*}
		\norm{(\theta_t \mathbb{1}) \cdot P_t}_{L^2_W \to L^2_W} \lesssim 1,
		\end{equation*}
		where the implicit 
		constant depends on $n, \lambda$ and $[W]_{A_2}$.
	\end{claim}
	\begin{proof}[Proof of the claim]
		This proof will follow that of \cite[(4.10)]{CUR}. Let us cover $\R^n$ by cubes $Q_k$ satisfying $t/2 < \ell(Q_k) \leq t$. In this way, we obtain 
		\begin{equation} \label{eq:split_Q_k}
		\int_{\R^n} \abs{(\theta_t \mathbb{1}) (x) \cdot P_t \mathbf{g}(x)}^2 W(x) dx
		=
		\sum_k \int_{Q_k} \abs{(\theta_t \mathbb{1}) (x) \cdot P_t \mathbf{g}(x)}^2 W(x) dx.
		\end{equation}
		We first establish an 
$L^\infty$ bound for $P_t \mathbf{g}(x)$, in the cube $Q_k$. 
Note that for $x \in Q_k$ we have that $P_t\mathbf{g}(x) = P_t(\mathbf{g} \mathbf{1}_{3Q_k})(x)$ because $t \leq 2 \ell(Q_k)$ (and $\supp \varphi \subset B(0, 1)$). 
Hence, 
		\begin{equation} \label{eq:pointwise_P_t}
		\begin{aligned}
		\abs{P_t(\mathbf{g} \mathbf{1}_{3Q_k})(x)}^2
		& \leq 
		\left( \norm{\varphi_t}_\infty \int_{3Q_k} 
		\abs{\mathbf{g}(y)} dy \right)^2
		\leq
		\left( t^{-n} \int_{3Q_k} \abs{\mathbf{g}(y)} dy \right)^2
		\\ & \approx 
		\left( \fint_{3Q_k} \abs{\mathbf{g}(y)} dy \right)^2
		\leq
		\frac{1}{\abs{3Q_k}^2} \int_{3Q_k} \abs{\mathbf{g}(y)}^2 W(y) dy \int_{3Q_k} W^{-1}(y) dy 
		\\ & =
		\frac{W(3Q_k)W^{-1}(3Q_k)}{\abs{3Q_k}^2} \fint_{3Q_k} \abs{\mathbf{g}(y)}^2 W(y) dy
		\lesssim
		\fint_{3Q_k} \abs{\mathbf{g}(y)}^2 W(y) dy
		\end{aligned}
		\end{equation}
		where we used the definition of an $A_2$ weight in the last step (see Definition~\ref{def:A_p}). 
		
We next claim that
	\begin{equation} \label{eq:theta_1_bounded}
		\frac{1}{W(Q_k)} \int_{Q_k} \abs{(\theta_t \mathbf{b}) (x)}^2 W(x) dx 
		\lesssim 
		\norm{\mathbf{b}}_\infty^2.
		\end{equation}
Taking this claim for granted momentarily,		
we obtain
		\begin{multline*}
		\int_{\R^n} \abs{(\theta_t \mathbb{1}) (x) \cdot P_t \mathbf{g}(x)}^2 W(x) dx 
		\lesssim
		\sum_k \int_{Q_k} \abs{(\theta_t \mathbb{1}) (x)}^2 \fint_{3Q_k} \abs{\mathbf{g}(y)}^2 W(y) dy \; W(x) dx
		\\ 
		\lesssim
		\sum_k \frac{1}{W(Q_k)} \int_{Q_k} \abs{(\theta_t \mathbb{1}) (x)}^2 W(x) dx \int_{3Q_k} \abs{\mathbf{g}(y)}^2 W(y) dy
		\\
		\lesssim 
		\sum_k \int_{3Q_k} \abs{\mathbf{g}(y)}^2 W(y) dy
		\lesssim 
		\int_{\R^n} \abs{\mathbf{g}(y)}^2 W(y) dy,
		\end{multline*}
		where in the last two steps we used first
		\eqref{eq:theta_1_bounded}, and then the bounded overlap property of the cubes $3Q_k$. 
		
		It remains to verify  \eqref{eq:theta_1_bounded}.  We dualize:
choose $\mathbf{h}=(h_1,h_2,...,h_n) \in L^2_W$, with $\supp \mathbf{h} \subset Q_k$, 
and write
		\begin{equation} \label{eq:local_nonlocal_theta_1}
		\begin{aligned}
		\int_\Rn \theta_t & \mathbf{b} (x)\cdot \mathbf{h}(x) W(x) dx  
		=
		\int_{\R^n} te^{-t^2 \L} \widetilde{\div} A \mathbf{b} (x) \cdot 
\mathbf{h}(x) W(x) dx 
		\\ & =
		\int_{\R^n} A \mathbf{b} (x) \cdot t \nabla e^{-t^2 L} \mathbf{h}(x) W(x) dx 
		\lesssim 
		\norm{\mathbf{b}}_\infty \int_{\R^n} 
\abs{t \nabla e^{-t^2 L} \mathbf{h}(x)} W(x) dx 
		\\ 
		& \leq
		\norm{\mathbf{b}}_\infty \left( \int_{2Q_k} \abs{t \nabla e^{-t^2 L} \mathbf{h}(x)} W(x) dx 
		+ \sum_{j = 2}^\infty \int_{2^jQ_k \setminus 2^{j-1}Q_k} \abs{t \nabla e^{-t^2 L} \mathbf{h}(x)} W(x) dx\right)
		\\
		& =:
		\norm{\mathbf{b}}_\infty \left( I^{(k)} + II^{(k)} \right).
		\end{aligned}
		\end{equation}
		
		For the first term we may simply compute, using Jensen's inequality and the boundedness of $t \nabla e^{-t^2L}$ from Lemma~\ref{lem:bounded},
		\begin{equation*}
		I^{(k)} 
		\lesssim 
		\left( W(Q_k) \int_{2Q_k} \abs{t \nabla e^{-t^2 L} \mathbf{h}(x)}^2 W(x) dx \right)^{1/2}
		\lesssim 
		\sqrt{W(Q_k)} \norm{\mathbf{h}}_{L^2_W}.
		\end{equation*}
		
		For the second term we use Jensen again, and later the off-diagonal estimates from Lemma~\ref{lem:off_diagonal} (taking advantage of $\ell(Q_k) \approx t$) to obtain
		\begin{multline*}
		II^{(k)}
		\lesssim
		\sum_{j = 2}^\infty \left( W(2^jQ_k) \int_{2^jQ_k \setminus 2^{j-1}Q_k} \abs{t \nabla e^{-t^2 L} \mathbf{h}(x)}^2 W(x) dx \right)^{1/2}
		\\
		\lesssim 
		\sum_{j = 2}^\infty \left( C_D^j W(Q_k) e^{-c4^j} \int_{Q_k} \abs{\mathbf{h}(x)}^2 W(x) dx \right)^{1/2}
		\lesssim 
		\sqrt{W(Q_k)} \norm{\mathbf{h}}_{L^2_W}.
		\end{multline*}
		
		With the estimates for $I^{(k)}$ and $II^{(k)}$, we can substitute back in \eqref{eq:local_nonlocal_theta_1} and obtain
		\begin{equation*}
		\int_\Rn \theta_t \mathbf{b} (x) \cdot\mathbf{h}(x) W(x) dx 
		\lesssim 
		\sqrt{W(Q_k)} \norm{\mathbf{b}}_\infty \norm{\mathbf{h}}_{L^2_W},
		\end{equation*}
		which after squaring gives \eqref{eq:theta_1_bounded} by duality, as desired.  This completes the proof of Claim \ref{claim:U_t_bounded_2}.
		\end{proof}
	
	
	\begin{claim} \label{claim:orthogonality_1}
	Suppose $s\leq t$. 
Then 
		\begin{equation*}
		\norm{U_t Q_s}_{L^2_W \to L^2_W} \lesssim  \frac{s}{t}\,.
		\end{equation*}
	\end{claim}
	\begin{proof}[Proof of the claim]
	Note that we have the pointwise estimate
	\begin{equation}\label{M2}
	|P_tQ_s f|\, \lesssim \,\frac{s}{t} \, \mathcal{M}^2 f\,, 
	\end{equation}
where $\mathcal{M}^2 := \mathcal{M}\circ \mathcal M$ is the iterated Hardy-Littlewood maximal operator (with respect to Lebesgue measure).  One may verify \eqref{M2} by a standard argument using the size estimates 
and compact support of the kernels of $P_t$ and $Q_s$, along with the smoothness of the former, and the cancellation property of the latter.
We omit the well-known details.  Since $\mathcal{M}$ is bounded on $L^2_W$
(recall that $W\in A_2$), we find,
with the aid of Claims~\ref{claim:U_t_bounded_1} and \ref{claim:U_t_bounded_2}, and \eqref{M2}:
		\begin{multline*}
		\norm{U_t Q_s}_{L^2_W \to L^2_W} \\[4pt]
		 \leq 
		\norm{\theta_t P_t}_{L^2_W \to L^2_W} \norm{P_t Q_s}_{L^2_W \to L^2_W} + \norm{(\theta_t \mathbb{1}) \cdot P_t}_{L^2_W \to L^2_W} \norm{P_t Q_s}_{L^2_W \to L^2_W} 
		\\[4pt] \lesssim 
		\norm{P_t Q_s}_{L^2_W \to L^2_W}
		\lesssim 
		 \frac{s}{t} \,.
		\end{multline*}
	\end{proof}
	
	
	\begin{claim} \label{claim:U_t_self_improves}
		We have, uniformly on $t$,
		\begin{equation*}
		\norm{U_t \mathbf{g}}_{L^2_W} \lesssim t \norm{\nabla \mathbf{g}}_{L^2_W}.
		\end{equation*}
	\end{claim}
	\begin{proof}
		The proof is inspired by \cite[Lemma 3.5]{AAAHK}, and in fact is similar in spirit to that of Lemma~\ref{lem:semigroup_difference_estimate}, relying strongly in a decomposition in subcubes of the right size to use Poincaré's inequality, and some boundedness and off-diagonal estimates. Nevertheless, let us show it in detail, because some parts of it will be reused later. Cover $\Rn$ by a grid of non-overlapping dyadic cubes $Q_k$ with sidelength $t/2 < \ell(Q_k) \leq t$. Using the easy fact that $U_t \mathbb{1} = 0$ we compute 
		\begin{multline}\label{abdef}
		\norm{U_t \mathbf{g}}_{L^2_W}^2
		=
		\sum_k \norm{U_t \left( \mathbf{g} - [\mathbf{g}]_{2Q_k}\right)}_{L^2_W(Q_k)}^2
		\\ \lesssim 
		\sum_k \norm{\theta_t P_t^2 \left( \mathbf{g} - [\mathbf{g}]_{2Q_k}\right)}_{L^2_W(Q_k)}^2
		+ \sum_k \norm{(\theta_t \mathbb{1}) \cdot P_t^2 \left( \mathbf{g} - [\mathbf{g}]_{2Q_k}\right)}_{L^2_W(Q_k)}^2
		=:
		A + B.
		\end{multline}
		
		Let us first deal with $A$, denoting $S_t := \theta_t P_t^2$ because we intend to reuse some computations later on. For each term in the series, simply using linearity and the triangle inequality
		\begin{multline} \label{eq:split_local_nonlocal}
		\norm{S_t \left( \mathbf{g} - [\mathbf{g}]_{2Q_k}\right)}_{L^2_W(Q_k)}
		\leq 
		\norm{S_t \left(\left( \mathbf{g} - [\mathbf{g}]_{2Q_k}\right) \mathbf{1}_{2Q_k} \right) }_{L^2_W(Q_k)}
		\\ 
		\quad + \sum_{j=1}^\infty \norm{S_t \left(\left( \mathbf{g} - [\mathbf{g}]_{2Q_k} \right) \mathbf{1}_{2^{j+1}Q_k \setminus 2^jQ_k} \right) }_{L^2_W(Q_k)}
		=:
		I^{(k)} + \sum_{j=1}^\infty II^{(k)}_j. 
		\end{multline}
		
		Using the boundedness of $S_t$ (in this case, this follows from Lemmas~\ref{lem:P_t_bounded} and \ref{lem:bounded}) and Poincaré's inequality, we deduce 
		\begin{equation}\label{eqik}
		I^{(k)} 
		\lesssim 
		\norm{\mathbf{g} - [\mathbf{g}]_{2Q_k}}_{L^2_W(2Q_k)}
		\lesssim 
		\ell(Q_k) \norm{\nabla \mathbf{g}}_{L^2_W(2Q_k)}
		\leq 
		t \norm{\nabla \mathbf{g}}_{L^2_W(2Q_k)}.
		\end{equation}
		
		And for the other terms, we can use the off-diagonal estimates for $S_t$ (in this case, this follows from Lemmas~\ref{lem:off_diagonal} and \ref{lem:off_diagonal_P_t_composition}), and taking advantage of $\ell(Q_k)\approx t$ and Poincaré, we obtain, similarly to the situation in Lemma~\ref{lem:semigroup_difference_estimate},
		\begin{multline*}
		\sum_{j=1}^\infty II^{(k)}_j 
		\lesssim 
		e^{-c 4^j } 
\norm{\mathbf{g} - [\mathbf{g}]_{2Q_k}}_{L^2_W(2^{j+1}Q_k)}
		\\ \lesssim 
		\sum_{j=1}^\infty e^{-c 4^j} \norm{\mathbf{g} - [\mathbf{g}]_{2^{j+1}Q_k}}_{L^2_W(2^{j+1}Q_k)}
		+ \sum_{j=1}^\infty \sum_{i=1}^j e^{-c 4^j} C_D^{(j-i)/2} \norm{\mathbf{g} - [\mathbf{g}]_{2^{i+1}Q_k}}_{L^2_W(2^{i+1}Q_k)}
		\\
		\lesssim 
		\sum_{j=1}^\infty e^{-\frac{c}{2} 4^j} 2^j \ell(Q_k) \norm{\nabla \mathbf{g}}_{L^2_W(2^{j+1}Q_k)}
		\lesssim 
		t \sum_{j=1}^\infty e^{-\frac{c}{4} 4^j} \norm{\nabla \mathbf{g}}_{L^2_W(2^{j+1}Q_k)}.
		\end{multline*}
		
		Thus, going back to \eqref{eq:split_local_nonlocal} we obtain 
		\begin{equation*}
		\norm{S_t (\mathbf{g} - [\mathbf{g}]_{2Q_k})}_{L^2_W(Q_k)} 
		\lesssim 
		t\norm{\nabla \mathbf{g}}_{L^2_W(2Q_k)} + t\sum_{j=1}^\infty e^{-\frac{c}{4} 4^j} \norm{\nabla  \mathbf{g}}_{L^2_W(2^{j+1}Q_k)},
		\end{equation*}
		and hence 
		\begin{equation*}
		A
		\lesssim 
		t^2 \sum_k \norm{\nabla  \mathbf{g}}_{L^2_W(2Q_k)}^2 
		+ t^2 \sum_k \left( \sum_{j=1}^\infty e^{-\frac{c}{4} 4^j} \norm{\nabla  \mathbf{g}}_{L^2_W(2^{j+1}Q_k)} \right)^2
		=:
		t^2 (A_1 + A_2).
		\end{equation*}
		By bounded overlap of the cubes $2Q_k$ we easily get 
		\begin{equation*}
		A_1 \lesssim \norm{\nabla  \mathbf{g}}_{L^2_W}^2.
		\end{equation*}
For the other term, we note that $|x-y|\lesssim 2^j \ell(Q_k) \approx 2^j t$, whenever $x\in Q_k$, and $y\in 2^{j+1}Q_k$.  We further note that
$W(Q_k) \approx W(B_t(x))$, for $x\in Q_k$, and that for all $x\in\Rn$,
\[
e^{-\frac{c}{4} 4^j}  \int_{|x-y|\lesssim 2^j t} W(B_t(x))^{-1} W(x) \,dx
\lesssim e^{-\frac{c}{8} 4^j}\,,
\]
 by the doubling property of $W$.
We now use these observations, along with Cauchy-Schwarz, Fubini's theorem, and the fact that the cubes $Q_k$ are non-overlapping, to obtain
		\begin{multline*}
		A_2 
		\leq
		\sum_k \left( \sum_{j=1}^\infty e^{-\frac{c}{4} 4^j} \right) \left( \sum_{j=1}^\infty e^{-\frac{c}{4} 4^j} \norm{\nabla  \mathbf{g}}_{L^2_W(2^{j+1}Q_k)}^2 \right)
		=
		\sum_k \sum_{j=1}^\infty e^{-\frac{c}{4} 4^j} \norm{\nabla  \mathbf{g}}_{L^2_W(2^{j+1}Q_k)}^2
 \\[4pt]
 \lesssim\, \sum_k \sum_{j=1}^\infty e^{-\frac{c}{4} 4^j} 
 \int_{Q_k} W(B_t(x))^{-1}W(x)  \int_{|x-y|\lesssim 2^j t} 
 |\nabla {\bf g}(y)|^2 W(y)\, dy dx 
\\[4pt]
=\, \sum_{j=1}^\infty e^{-\frac{c}{4} 4^j} 
\int_{\Rn} W(B_t(x))^{-1}W(x)  \int_{|x-y|\lesssim 2^j t} |\nabla {\bf g}(y)|^2 W(y)\, dy dx 
\\[4pt]
\lesssim\,
 \sum_{j=1}^\infty e^{-\frac{c}{8} 4^j} 
\int_{\Rn}  |\nabla {\bf g}(y)|^2 W(y)\, dy  		
		\,\lesssim \,
		\norm{\nabla  \mathbf{g}}_{L^2_W}^2.
		\end{multline*}
Consequently, we have shown that
		\begin{equation*}
		A \lesssim t^2 \norm{\nabla  \mathbf{g}}_{L^2_W}^2.
		\end{equation*}
		
We can apply a similar, but simpler argument to handle term $B$ in \eqref{abdef}.  We now set
$S_t := (\theta_t \mathbb{1}) \cdot P_t^2$, and 
note that $S_t$ is uniformly bounded on $L^2_W$,
by Claim \ref{claim:U_t_bounded_2} and Lemma \ref{lem:P_t_bounded}.
Moreover, the kernel of $P^2_t$ is compactly supported in the ball of radius $2t$, so the same is true for $S_t$.  Hence, for the current version of $S_t$, we obtain a simplified variant of \eqref{eq:split_local_nonlocal}, in which only the term $I^{(k)}$ appears, enjoying the same bound 
as in \eqref{eqik}.  Thus,
		\begin{equation*}
		B \lesssim t^2 \norm{\nabla  \mathbf{g}}_{L^2_W}^2\,.
		\end{equation*}
The proof of Claim \ref{claim:U_t_self_improves} is now complete.
	\end{proof}
	
	
	\begin{claim}\label{claim4.26}
		For $t\leq s$, we have
		\begin{equation*}
		\norm{U_t Q_s}_{L^2_W \to L^2_W} \lesssim \frac{t}{s}.
		\end{equation*}
	\end{claim}
	\begin{proof}[Proof of the claim]
		Using Claim~\ref{claim:U_t_self_improves} and 
		Lemma \ref{lem:P_t_bounded}, we have
		\begin{equation*}
		\norm{U_t Q_s \mathbf{g}}_{L^2_W}
		\lesssim 
		t \norm{\nabla Q_s \mathbf{g}}_{L^2_W \to L^2_W}
		=
		\frac{t}{s} \norm{s\nabla Q_s\mathbf{g}}_{L^2_W}
		\lesssim 
		\frac{t}{s} \norm{\mathbf{g}}_{L^2_W},
		\end{equation*}
		as desired.
		
	\end{proof}
	
	As noted above, the preceding claims conclude the proof of
	Lemma \ref{lem:difference_term_LP}.
	\end{proof}


We are now ready to reduce matters to a Carleson measure estimate.
Recall that to prove Theorem~\ref{th:sq_fun_adj} (and hence Theorem~\ref{th:kato_adj}), it suffices to verify estimate 
 \eqref{goal:sq_fun_theta} (equivalently, \eqref{goal:sq_fun}).

\begin{lemma} \label{lem:reduction_CME}
	Theorem~\ref{th:sq_fun_adj} (and hence Theorem~\ref{th:kato_adj}) follows from the Carleson measure estimate 
	\begin{equation} \label{goal:CME}
	\sup_Q \frac{1}{W(Q)} \int_0^{\ell(Q)} \int_Q \abs{\ttheta_t \mathbb{1}(x)}^2 W(x) \frac{dxdt}{t} < \infty.
	\end{equation}
\end{lemma}
\begin{proof}
	Recalling that $\ttheta_t \mathbb{1} = \theta_t \mathbb{1}$, we see that
by Lemma~\ref{lem:difference_term_LP} and a weighted version of Carleson's embedding inequality (see \cite[Lemma 2.2]{CUR}), 
the estimate \eqref{goal:CME} 
implies \eqref{goal:sq_fun_theta}. 
\end{proof}


Our goal then, is to prove \eqref{goal:CME}.  To this end, let us first 
establish a few more estimates to be used in the sequel. 
We define the dyadic averaging operator by 
\begin{equation*}
A_t \mathbf{f}(x) := \fint_{Q_{x, t}} \mathbf{f}(y) dy,
\end{equation*} 
where $Q_{x, t}$ is the half-open dyadic cube containing $x$ for which $t/2 < \ell(Q_{x, t}) \leq t$.

\begin{lemma} \label{lem:difference_averaging}
	We have
	\begin{equation*}
	\int_0^\infty \norm{ 
		(\theta_t \mathbb{1}) \cdot (P_t^2 - A_t) \mathbf{g} 
	}^2_{L^2_W} \frac{dt}{t} 
	\lesssim 
	\norm{\mathbf{g}}_{L^2_W}^2,
	\end{equation*}
	where the implicit constant depends on $n, \lambda$ and $[W]_{A_2}$.
\end{lemma}
\begin{proof}
	The proof of this estimate will be very similar to that of Lemma~\ref{lem:difference_term_LP}. We set 
	\[
\widetilde{U}_t := (\theta_t \mathbb{1}) \cdot (P_t^2 - A_t)\,,
\] 
and note that it is enough to show that $\widetilde{U}_t$ satisfies 
the hypotheses of the weighted Littlewood-Paley almost-orthogonality result 
Lemma~\ref{lem:orthogonality}. The uniform boundedness of $\widetilde{U}_t$ arises immediately from that of
$(\theta_t \mathbb{1}) \cdot P_t^2$ (see Claim \ref{claim:U_t_bounded_2}
and Lemma~\ref{lem:P_t_bounded}),
along with the following result:
	
	
	\begin{claim} \label{claim:A_t_bounded}
		We have, uniformly on $t$,
		\begin{equation*}
		\norm{(\theta_t \mathbb{1}) \cdot A_t}_{L^2_W \to L^2_W} \lesssim 1.
		\end{equation*}
	\end{claim}
	\begin{proof}[Proof of the claim]
	The proof is the same as that of Claim~\ref{claim:U_t_bounded_2}, which
	treated  $(\theta_t \mathbb{1}) \cdot P_t$.  Indeed, the only properties of 
$P_t$ that were used in that argument were the size and support condition of its kernel.  The kernel of $A_t$ enjoys similar properties, in fact
		\begin{equation*}
		\abs{A_t\mathbf{f}(x)}^2
		\leq 
		\left( \fint_{Q_k} \abs{\mathbf{f}(y)} dy \right)^2
		\lesssim 
		\fint_{Q_k} \abs{\mathbf{f}(y)}^2 W(y) dy,
		\end{equation*} 
	hence, the same proof may be repeated.
	\end{proof}
	
	
	To prove the quasi-orthogonality with the $Q_s$ operators,
	 the next result will be useful.
	
	\begin{claim} \label{claim:tilde_U_t_self_improves}
		We have, uniformly on $t$,
		\begin{equation*}
		\norm{\widetilde{U}_t \mathbf{g}}_{L^2_W} \lesssim t \norm{\nabla \mathbf{g}}_{L^2_W}.
		\end{equation*}
	\end{claim}
	\begin{proof}
		The proof is similar to that of Claim~\ref{claim:U_t_self_improves}, but simpler: now one has to deal only with terms like ``$B$" associated to $S_t = (\theta_t \mathbb{1}) \cdot A_t$ in \eqref{abdef}, so that there is no ``tail" as in \eqref{eq:split_local_nonlocal}, but rather
only a local term analogous to $I^{(k)}$.  We omit the routine details.
\end{proof}	
	
	
	The following two claims finish the proof of Lemma
	\ref{lem:difference_averaging}, and are analogous to those in the proof of Lemma~\ref{lem:difference_term_LP}.
	
	\begin{claim}
		We have, uniformly for $t\leq s$,
		\begin{equation*}
		\norm{\widetilde{U}_t Q_s}_{L^2_W \to L^2_W} \lesssim \frac{t}{s}.
		\end{equation*}
	\end{claim}
	\begin{proof}
		In view of Claim~\ref{claim:tilde_U_t_self_improves}, repeating the proof of Claim \ref{claim4.26}, we simply write 
		\begin{equation*}
		\norm{\widetilde{U}_t Q_s}_{L^2_W \to L^2_W}
		\lesssim 
		t \norm{\nabla Q_s}_{L^2_W \to L^2_W}
		\lesssim 
		\frac{t}{s}\,.
		\end{equation*}
	\end{proof}
	
	
	\begin{claim}
		We have, uniformly in $s\leq t$, and for some fixed $\alpha > 0$,
		\begin{equation*}
		\norm{\widetilde{U}_t Q_s}_{L^2_W \to L^2_W} \lesssim \left( \frac{s}{t} \right)^\alpha.
		\end{equation*}
	\end{claim}
	\begin{proof}
		On the one hand, as in Claim~\ref{claim:orthogonality_1}, and using the boundedness of Claim~\ref{claim:U_t_bounded_2},
		\begin{equation*}
		\norm{(\theta_t \mathbb{1}) \cdot P_t^2 Q_s}_{L^2_W \to L^2_W}
		\lesssim 
		\norm{P_t Q_s}_{L^2_W \to L^2_W}
		\lesssim 
		\frac{s}{t}.
		\end{equation*}
On the other hand, by
	 \cite[Lemma 4.7 and its proof]{AHLMcT}, we have the unweighted
quasi-orthogonality estimate
\[\norm{A_t Q_s}_{L^2 \to L^2} \lesssim \left(\frac{s}{t}\right)^\alpha\,,
\]
for some exponent $\alpha >0$, uniformly for $s\leq t$.  Consequently, 
we may use 
 the technique of Duoandikoetxea and Rubio de Francia \cite{DuoRdF}, 
 in which one first self-improves the weight $W$, and then uses 
 Stein-Weiss interpolation with change of measure \cite{SW}, to deduce the weighted quasi-orthogonality estimate
 \[\norm{A_t Q_s}_{L_W^2 \to L_W^2} \lesssim \left(\frac{s}{t}\right)^{\beta}\,,
\]
for some positive $\beta < \alpha$
 (see Lemma 2.5 in \cite{CUR} for more details).
 Hence, by Claim~\ref{claim:A_t_bounded}, 
		\begin{equation*}
		\norm{(\theta_t \mathbb{1}) \cdot A_t Q_s}_{L^2_W \to L^2_W}
		= 
		\norm{(\theta_t \mathbb{1}) \cdot A_t^2 Q_s}_{L^2_W \to L^2_W}
		\lesssim 
		\norm{A_t Q_s}_{L^2_W \to L^2_W}
		\lesssim 
		\left( \frac{s}{t} \right)^{\beta}.
		\end{equation*}
	\end{proof}
Collecting all the above claims, the proof of Lemma \ref{lem:difference_averaging} is completed.
	\end{proof}


\begin{corollary}\label{cor-diff}
We have the square function bound
\[
\int_0^\infty \norm{ 
	\ttheta_t  \nabla f - (\ttheta_t \mathbb{1}) \cdot A_t \nabla f
	}^2_{L^2_W} \frac{dt}{t} 
	\lesssim 
	\norm{\nabla f}_{L^2_W}^2,
\]
where the implicit constant depends on $n, \lambda$ and $[W]_{A_2}$.
\end{corollary}
\begin{proof}
With Lemma \ref{lem:difference_averaging} in hand, since 
$\ttheta_t \mathbb{1} = \theta_t \mathbb{1}$, it is enough to prove the following:
	\begin{equation*}
	\int_0^\infty \norm{ 
		\ttheta_t \nabla f - (\theta_t \mathbb{1}) \cdot P^2_t \nabla f
	}^2_{L^2_W} \frac{dt}{t} 
	\lesssim 
	\norm{\nabla f}_{L^2_W}^2\,.
	\end{equation*}
To this end, we write
\[
\ttheta_t \nabla f - (\theta_t \mathbb{1}) \cdot P^2_t \nabla f
\,=\, \ttheta_t \nabla (I-P^2_t) f  + \left[ \ttheta_t \nabla P^2_t f 
- (\theta_t \mathbb{1}) \cdot P^2_t \nabla f\right] \, =: Y_t f + Z_t f\,.
\]
By \eqref{thetaidentity},
\[
-2 Y_t = t\L e^{-t^2\L} \,(I-P_t^2) =:\widetilde{R}_t\,,
\]
where $\widetilde{R}_t$ is precisely the same operator defined in \eqref{eq:decomposition_adj}, enjoying the square function bound established in Lemma \ref{lemma4.8}.  In addition, again using \eqref{thetaidentity},
\[
-2\, \ttheta_t \nabla P^2_t = t\L e^{-t^2\L} \,P_t^2 =: \widetilde{T}_t\,,
\]
where $\widetilde{T}_t$ is precisely the same operator defined in
\eqref{eq:decomposition_adj}.  We now repeat the splitting of $\widetilde{T}_t$, exactly as in \eqref{Ttsplit}:
\[
\widetilde{T}_t f=
- t e^{-t^2\L} L P_t^2 f - 2 t e^{-t^2\L} \widetilde{\div} (A \nabla (P_t^2 f))\,.
\]
Note that the the second term equals
$-2 \theta_t \nabla P^2_t f$ 
(see \eqref{thetadef}).  
Combining these observations, we see that
\[
Z_tf \,=\, -\frac12 \widetilde{T}_t f-  (\theta_t \mathbb{1}) \cdot P^2_t \nabla f
\,=\, \frac12 t e^{-t^2\L} L P_t^2 f\, + \, \left[\theta_t  P^2_t \nabla f -  (\theta_t \mathbb{1}) \cdot P^2_t \nabla f\right]=: E_tf + U_tf\,,
\]
where the term $E_tf$ (which is actually the error $(\ttheta_t -\theta_t)\nabla P_t^2 f$), satisfies the desired square function bound, by Lemma
\ref{lemma4.10}.
The last term also enjoys the desired square function bound, by Lemma \ref{lem:difference_term_LP}.
This concludes the proof of the Corollary.
\end{proof}


\subsection{The $T(b)$ argument}

Recall that our goal is to prove the Carleson measure estimate \eqref{goal:CME}.  
We now turn to this task,
which will finish the proof of Theorem~\ref{th:kato_adj} (and therefore also the proof of Theorem~\ref{th:kato}).  Our arguments here will be an adaptation
of the proof of the Kato conjecture in the divergence form setting, see
\cite{AHLMcT}, and in particular, the extension of that proof to the degenerate
elliptic case in \cite{CUR}.

We note that by the doubling property of $W$, we may assume that the supremum in \eqref{goal:CME} is taken over dyadic cubes $Q$.  Given any such cube $Q$, a sufficiently small number $\varepsilon \in (0, 1)$ to be chosen, and $v \in \Rn$ 
with $\abs{v} = 1$, we define 
\begin{equation} \label{eq:def_f}
f_{Q, v}^\varepsilon 
:=
e^{-(\varepsilon \ell(Q))^2 \L} ({\bf \Phi}_Q \chi_Q \cdot v),
\end{equation}
where ${\bf \Phi}_Q(x) = x - x_Q$, $x_Q$ denotes the center of $Q$, and $\chi_Q \in \mathscr{C}_0^\infty$ is a cut-off function such that $\chi_Q \equiv 1$ in $2Q$, $\supp \chi_Q \subset 4Q$ and $\norm{\chi_Q}_\infty + \ell(Q) \norm{\nabla \chi_Q}_\infty + \ell(Q)^2 \norm{\nabla^2 \chi_Q}_\infty \lesssim 1$. Clearly, 
\begin{equation} \label{eqlipnorm}
\norm{\nabla ({\bf \Phi}_Q \chi_Q \cdot v)}_\infty \lesssim 1\,,
\end{equation}
and also
\begin{equation} \label{phil2norm}
\int_{\Rn} |{\bf \Phi}_Q \chi_Q \cdot v|^2 W(x) dx\, \lesssim \,\ell(Q)^2 W(Q)\,.
\end{equation}

The following estimates hold for $f_{Q, v}^\varepsilon$, with constants that are uniform on $Q, v$ and $\varepsilon$: 
\begin{equation} \label{eq:5.2}
\int_{5Q} \abs{f_{Q, v}^\varepsilon - {\bf \Phi}_Q \chi_Q \cdot v}^2 W dx 
\lesssim 
\varepsilon^2 \ell(Q)^2 W(Q),
\end{equation}
\begin{equation} \label{eq:5.3}
\int_{5Q} \abs{\nabla f_{Q, v}^\varepsilon }^2 W dx \,+\,
\int_{5Q} \abs{\nabla \left(f_{Q, v}^\varepsilon - {\bf \Phi}_Q \chi_Q \cdot v\right)}^2 W dx 
\,\lesssim \, 
W(Q),
\end{equation}
These estimates follow at once from \eqref{eqlipnorm}, Lemmas~\ref{lem:semigroup_difference_estimate} and \ref{lem:gradient_semigroup_difference_estimate} 
(with $t=\varepsilon \ell(Q)$), and the doubling property of $W$.


The proof of \eqref{goal:CME} (and hence of Theorem~\ref{th:kato_adj} by Corollary~\ref{lem:reduction_CME}) follows from the next two lemmas.

\begin{lemma}\label{tbexist}
	There exists $0 < \varepsilon = \varepsilon(\lambda, n, [W]_{A_2}) \ll 1$ and a finite set $V$ of unit vectors in $\Rn$,
whose cardinality depends only on $\varepsilon$ and $n$, such that 
	\begin{multline*}
	\sup_Q \frac{1}{W(Q)} \int_0^{\ell(Q)} \int_Q \abs{ (\ttheta_t \mathbb{1}) (x) }^2 W(x) \frac{dxdt}{t} 
	\\ \lesssim
	\sum_{v \in V} \sup_Q \frac{1}{W(Q)} \int_0^{\ell(Q)} \int_Q \abs{ (\ttheta_t \mathbb{1}) (x) \cdot \left( A_t \nabla f_{Q, v}^\varepsilon\right)(x) }^2 W(x) \frac{dxdt}{t},
	\end{multline*}
	where the implicit constant depends on $n, \lambda$ and $[W]_{A_2}$.
\end{lemma}
\begin{proof} 
	The reader may check that the proof of \cite[Lemma 5.1]{CUR} (which in turn is an adaptation to the weighted case of
	\cite[Lemma 5.4]{AHLMcT}) works perfectly well in our situation: as long as $W \in A_2$ and $f_{Q, v}^\varepsilon$ satisfies the estimates \eqref{eq:5.2} and \eqref{eq:5.3}, the proof in \cite{CUR} goes through\footnote{To clarify a possible point of confusion, we mention that in \cite{AHLMcT} and \cite{CUR}, the unit vectors were taken in $\mathbb{C}^n$, because in the divergence form setting of those papers, one treats the case of complex coefficients;  at present, our results in the non-divergence form case treat only the case of real coefficients, so we need only consider real unit vectors.}.
\end{proof}


With Lemma \ref{tbexist} in hand, estimate \eqref{goal:CME} will follow immediately from the next lemma.
\begin{lemma}
	For every cube $Q$ and unit vector $v$, we have
	\begin{equation*}
	\int_0^{\ell(Q)} \int_Q \abs{ (\ttheta_t \mathbb{1})(x) \cdot \left( A_t \nabla f_{Q, v}^\varepsilon \right) (x) }^2 W(x) \frac{dxdt}{t}
	\lesssim 
	W(Q),
	\end{equation*}
	where the implicit constant depends on $n, \lambda, [W]_{A_2}$ and $\varepsilon$, but is uniform on $Q$ and $v$.
\end{lemma}
\begin{proof}
	Fix $Q$ and $v$, and abbreviate $f := f_{Q, v}^\varepsilon$. 
By Corollary \ref{cor-diff}, we have
	\begin{multline*}
	\int_0^{\ell(Q)} \int_Q \abs{ (\ttheta_t \mathbb{1})(x) \cdot \left( A_t \nabla f \right) (x) }^2 W(x) \frac{dxdt}{t}
	\\ 
	\lesssim \,
	\norm{\nabla f}_{L^2_W}^2  
	+ \int_0^{\ell(Q)} \int_Q \abs{ (\ttheta_t \nabla f ) (x) }^2 W(x) \frac{dxdt}{t}\,
	=: \,
	I + II\, \lesssim \,W(Q) + II\,,
	\end{multline*}
	where in the last step,
	we have used \eqref{eq:5.3} to obtain the desired bound for term $I$.
	
	Term $II$ can be treated as follows, using  \eqref{thetaidentity}, 
	Lemma~\ref{lem:bounded},  and
the definition of 
$f =f_{Q,v}^{\varepsilon}$, 
	\begin{multline*}
II\, \approx\,
	\int_0^{\ell(Q)} \int_Q \abs{ t e^{-t^2 \L } \L f (x) }^2 W(x) \frac{dxdt}{t}\\[4pt]
	\lesssim \,
	\int_0^{\ell(Q)} t dt \int_\Rn \abs{\L e^{-(\varepsilon \ell(Q))^2\L} ({\bf \Phi}_Q \chi_Q \cdot v)(x)}^2 W(x) dx
	\\[4pt]
	 \approx
	\ell(Q)^{2} \int_\Rn \abs{ \L
	e^{-(\varepsilon \ell(Q))^2\L}  ({\bf \Phi}_Q \chi_Q \cdot v)(x)}^2 W(x) dx
\\[4pt]
	 \lesssim	
\varepsilon^{-4} \ell(Q)^{-2}\int_{\Rn} |{\bf \Phi}_Q \chi_Q \cdot v|^2 W(x) dx
\, \lesssim \varepsilon^{-4} \, W(Q)\,,
	\end{multline*}
where in the last two steps we have first used Lemma~\ref{lem:bounded} 
\ref{bounded:L_semigroup_adjoint} with $t =\varepsilon \ell(Q)$, and then 
\eqref{phil2norm}.  Since $\varepsilon$ has been fixed depending only on allowable parameters, the dependence on $\varepsilon$ is harmless.	
	
	Collecting all the preceeding estimates, we have finished the proof.
\end{proof}


\end{document}